\documentclass[10pt]{article}
\usepackage{color}
\usepackage{array}
\usepackage{algpseudocode}
\usepackage{algorithm}
\usepackage{amssymb}   
\usepackage{amsfonts}
\usepackage{amsthm}    
\usepackage{amsmath}   
\usepackage{stmaryrd}  
\usepackage{titletoc}  
\usepackage{mathrsfs}  
\usepackage{graphicx}
\usepackage{hyperref}
\usepackage[normalem]{ulem}
\usepackage{cancel}

\usepackage{todonotes}

\newlength{\defbaselineskip}
\newcommand{\setlinespacing}[1]%
           {\setlength{\baselineskip}{#1 \defbaselineskip}}

\theoremstyle{plain}
\newtheorem{thm}{Theorem}[section]
\newtheorem{lem}[thm]{Lemma}
\newtheorem{prop}[thm]{Proposition}
\newtheorem{cor}[thm]{Corollary}

\theoremstyle{definition}
\newtheorem{defn}[thm]{Definition}
\newtheorem{ass}[thm]{Assumption}
\newtheorem{rmk}[thm]{Remark}
\newtheorem{example}[thm]{Example}

\newcommand{\RN}[1]{%
  \textup{\uppercase\expandafter{\romannumeral#1}}%
}

\DeclareMathOperator*{\esssup}{esssup}
\DeclareMathOperator*{\essinf}{essinf}
\newcommand{\ud}{\mathrm{d}}

\newcommand{\cH}{\mathcal{H}}

\newcommand{\cB}{\mathcal{B}}
\newcommand{\cA}{\mathcal{A}}

\newcommand{\bE}{\mathbb{E}}

\newcommand{\bP}{\mathbb{P}}
\newcommand{\bR}{\mathbb{R}}

\newcommand{\sF}{\mathscr{F}}

\usepackage{float}
\makeatletter\@addtoreset{equation}{section} \makeatother
 \allowdisplaybreaks
\begin{document}

\title{Optimal Resource Extraction under Distribution Learning and Infinite-Horizon Stochastic Hamilton--Jacobi Equations\thanks{The first and third author gratefully acknowledge support from the DFG CRC/TRR 388 ``Rough Analysis, Stochastic Dynamics and Related Fields'' -- Project ID 516748464. The authors also acknowledge the support of the Banff International Research Station (BIRS) for the Workshop [24w5257] ``Modeling, Learning and Understanding: Modern Challenges between Financial Mathematics, Financial Technology and Financial Economics'', November 10--15, 2024, where part of this work was carried out.}}

\author{ Ulrich Horst\footnotemark[1] \and Jinniao Qiu\footnotemark[2] \and Yang Yang\footnotemark[3]  }
\footnotetext[1]{Department of Mathematics and School of Business and Economics, Humboldt-Universität zu Berlin, Unter den Linden 6, 10099 Berlin, Germany. \textit{E-mail}: \texttt{horst@math.hu-berlin.de} (U. Horst).}
\footnotetext[2]{Department of Mathematics \& Statistics, University of Calgary, 2500 University Drive NW, Calgary, AB T2N 1N4, Canada. \textit{E-mail}: \texttt{jinniao.qiu@ucalgary.ca} (J. Qiu).}
\footnotetext[3]{Department of Mathematics, Humboldt-Universität zu Berlin, Unter den Linden 6, 10099 Berlin, Germany. \textit{E-mail}: \texttt{yang.yang.1@hu-berlin.de} (Y. Yang).}

\maketitle

\begin{abstract} 
We study an infinite-horizon stochastic control problem for the optimal exploitation of an exhaustible resource with unknown total reserves. Information is generated both endogenously through continued extraction without depletion and exogenously through an external information flow. This interaction makes the natural problem non-Markovian and time-inconsistent. We show that it nevertheless admits an equivalent time-consistent formulation with the same optimal controls. The associated value function is characterized as the unique viscosity solution of a stochastic Hamilton--Jacobi equation with random coefficients. We prove comparison on the infinite horizon through a Snell-envelope-based strict-contact argument and establish uniqueness by an independent Brownian regularization and a BSDE correction, avoiding piecewise Markovian approximations. Finally, we identify the deterministic benchmark and show that, under persistent reserve uncertainty, the rescaled stochastic value function converges to a pathwise deterministic control problem, with the long-run optimal extraction rate determined by the asymptotic hazard rate of the limiting reserve distribution.
\end{abstract}



\section{Introduction}\label{sec:introduction}

The optimal exploitation of exhaustible resources is a classical problem in economics, going back to Hotelling's analysis of the intertemporal trade-off between current extraction and the scarcity value of resources left in the ground \cite{Hotelling1931}. A central difficulty is that the size of the economically recoverable resource stock is typically not known when extraction decisions are made. Reserve uncertainty matters because extraction and exploration not only deplete known reserves but also generate information about the remaining stock. This interaction between depletion and learning was already central to the early literature. Loury \cite{Loury1978} studies the optimal exploitation of an unknown reserve, Hoel \cite{Hoel1978} analyzes how information revealed through extraction affects production under reserve uncertainty, and Swierzbinski and Mendelsohn \cite{SwierzbinskiMendelsohn1989} develop a continuous-time Bayesian model of information acquisition about an exhaustible resource. More recently, Jakobsson et al.~\cite{JakobssonEtAl2012} use Bayesian updating to study how exploration changes beliefs about the distribution and total size of oil resources. In general, such beliefs are formed from an accumulating history of observations, 
making a genuinely path-dependent information structure economically natural rather than a purely technical generalization.

The informational role of exploration is closely connected to the stochastic dynamics of reserve discovery studied in the more recent mathematical literature. Ludkovski and Sircar \cite{LudkovskiSircar2012,LudkovskiSircar2015}, for example, model exploration as a controlled stochastic discovery mechanism that replenishes existing reserves. This provides a natural link to more recent work on strategic interaction, depletion, and market exit. Chan and Sircar \cite{ChanSircar2015,ChanSircar2017} analyze Bertrand and Cournot mean field games and extensions involving shale production and renewable substitutes, Graber and Bensoussan \cite{GraberBensoussan2018} consider mean field games with absorption, and Graewe, Horst, and Sircar \cite{GraeweHorstSircar2022} study deterministic mean field games of resource exploitation with depletion-induced exit. Related exit mechanisms also arise in portfolio liquidation and market-entry models; see \cite{HorstDropOut,HorstEntryLiquidation}.

Against this background, we return to the problem of uncertain total reserves within a modern stochastic-control framework. We consider a single producer who knows the distribution of total reserves but not their realization, and whose extraction activity ends once the resource is exhausted. In addition to learning from continued extraction, the producer receives an exogenous flow of information that changes her assessment of the available stock. The model therefore combines two distinct learning mechanisms. {\sl Endogenous learning} occurs through depletion: every additional unit extracted without exhaustion reveals that the stock must be larger than previously known. {\sl Exogenous learning} is driven by information outside the extraction decision, such as geological surveys, exploratory drilling, or reports about nearby reserves. We model this information through a filtration generated by a Wiener process, while allowing reserve beliefs to depend on the entire history of observed signals. Thus, even though the underlying information flow is Brownian, the resulting control problem may be genuinely non-Markovian.

The two forms of learning affect the control problem in fundamentally different ways. Exogenous information makes the coefficients of the associated Hamilton--Jacobi (HJ) equation random, while endogenous learning changes the producer's conditional distribution of reserves as extraction proceeds. Continued extraction without depletion leads the producer to condition on the event that the true resource stock exceeds the amount already extracted. Since this conditioning event changes with the state, the usual dynamic-programming structure breaks down and the original objective becomes time-inconsistent. In the present model, however, this time inconsistency has a particularly simple structure: the normalization induced by conditioning changes the value of future revenues, but not the ranking of admissible controls at a fixed time and state. We exploit this observation to construct an equivalent time-consistent problem with the same optimal strategies. Thus, in the multi-person interpretation of time-inconsistent control, successive selves share the same preference ordering. This connects our setting to the equilibrium-control literature \cite{ekeland2010golden,wei2017time}, while allowing us to recover a standard optimization problem rather than replacing optimal controls by equilibrium controls. The resulting value function is characterized by a stochastic HJ equation with random, potentially path-dependent coefficients. Although this equation is not explicitly solvable, its long-run behavior admits a remarkably explicit characterization: after rescaling, the value function converges to a pathwise deterministic limit, and the optimal extraction rate is determined explicitly by the tail hazard rate of the limiting reserve distribution. Thus, residual uncertainty about large reserve levels directly governs long-run production.

The value function of the resulting time-consistent problem is characterized by an infinite-horizon stochastic Hamilton--Jacobi equation. Its viscosity analysis is more delicate than in the deterministic setting. A pathwise contact point between a candidate solution and a test function may depend on future randomness and therefore fail to be a stopping time. Following \cite{qiu_stochastic_hjb,qiu_uniqueness_2019}, we formulate tangency instead through conditional expectations and optimal stopping. At the same time, stochastic test functions are semimartingale random fields and contain martingale components in addition to spatial derivatives and drift. Classical doubling-of-variables arguments control spatial jets but not these martingale terms, so the standard deterministic comparison proof cannot be transferred directly; see \cite{ekren2016viscosity-1,ekren2016viscosity-2,qiu_stochastic_hjb,qiu_uniqueness_2019}. Our comparison proof therefore uses the stochastic strict-contact idea of \cite{qiu_stochastic_hjb,qiu_uniqueness_2019}: a Snell-envelope argument yields a finite stopping time at which stochastic tangency can be established, and a martingale correction produces the required test function. The proof requires neither the strong regularity assumptions of \cite{qiu_stochastic_hjb} nor the penalization approximations of \cite{qiu_uniqueness_2019}. The infinite horizon creates the additional possibility that a violation of the desired ordering escapes to temporal or spatial infinity, which we exclude through suitable one-sided asymptotic conditions.

The uniqueness argument is completed by a regularization procedure that avoids the piecewise Markovian approximations used in related stochastic HJ theories \cite{ekren2016viscosity-1,ekren2016viscosity-2,qiu_uniqueness_2019}. We add a small independent Brownian perturbation to the state dynamics, which turns the auxiliary equations into uniformly superparabolic stochastic PDEs and provides the spatial regularity needed for comparison. After an appropriate spatial shift and a linear BSDE correction, the corresponding value functions yield regular upper and lower barriers whose gap vanishes with the perturbation. This gives uniqueness by a squeezing argument and, at the same time, preserves the genuinely non-Markovian structure of the original problem. 

We finally turn to the deterministic benchmark and to the long-run effect of learning. When the conditional survival function is deterministic and time-independent, the value function separates into a discount factor and a stationary component. We identify the latter both as the unique classical solution of a first-order boundary-value problem and as the value function of the corresponding deterministic control problem. In the general model, conditional survival probabilities converge pointwise as information accumulates. Under {\sl uniform} convergence and additional regularity assumptions, the appropriately rescaled stochastic value function converges to the pathwise classical value function associated with the limiting information. If, in addition, the limiting hazard rate converges for large resource levels, then the optimal extraction rate converges as well, with an explicit limit determined by that hazard rate. The tail behavior of residual reserve uncertainty has a direct economic meaning: it determines the producer's long-run extraction behavior.

In summary, we formulate a resource-extraction problem with unknown total reserves that combines endogenous learning through depletion with exogenous learning through a potentially path-dependent Brownian information flow. We establish existence, comparison, and uniqueness for the associated infinite-horizon stochastic Hamilton--Jacobi equation and characterize the long-run behavior of the value function and optimal feedback. Section~\ref{sec:model} introduces the model and its time-consistent reformulation, Section~\ref{section_sto_con} develops the stochastic viscosity theory, and Section~\ref{subsection_special_case} studies the deterministic benchmark and asymptotic behavior.


\medskip \textbf{Frequently used notation.} Let $C(\bR^+)$ denote the space of continuous functions on $\bR^+$, and let $C_0(\bR^+)$ denote the subspace of functions satisfying $\lim_{x\rightarrow \infty} f(x)=0$.
For $p\in [1,\infty)$, let $\mathcal{S}^p(C_0(\bR^+))$ be the space of $(\sF_t)_{t\geq 0}$-adapted $C_0(\bR^+)$-valued processes $u$ such that, for a.e.~$\omega\in\Omega$, $u(\omega)$ belongs to $C(\bR^+;C_0(\bR^+))$ and
\begin{equation*}
\| u \|_{\mathcal{S}^p(C_0(\bR^+))}:=\left\| \sup_{(t,x)\in \bR^+\times\bR^+} | u(t,x) | \right\|_{L^p(\Omega, \sF, \bP)}<\infty.
\end{equation*}
We define $\mathcal S^p(C_0(\bR^+)\oplus \bR)$ as the space of functions $v: \Omega\times \bR^+\times\bR^+\to \bR$ admitting a unique decomposition 
\[
v(t,x) = v_0(t,x) + v_{\infty}(t),
\]
where $v_0\in \mathcal S^p(C_0(\bR^+))$ and $v_{\infty}$ is a real-valued $(\sF_t)_{t\geq 0}$-adapted continuous process such that
\[
\| v \|_{\mathcal S^p(C_0(\bR^+)\oplus \bR)} := \|v_0\|_{\mathcal S^p(C_0(\bR^+))} + \left\| \sup_{t\ge 0} |v_{\infty}(t)| \right\|_{L^p(\Omega,\sF,\bP)}<\infty.
\]
For $p\in [1,\infty)$, let $\mathcal{L}^p(C_0(\bR^+))$ 
denote the space of $(\sF_t)_{t\geq 0}$-adapted $C_0(\bR^+)$-valued processes $\mathcal{X}$ such that
\begin{align*}
\|\mathcal{X}\|_{\mathcal{L}^p(C_0(\bR^+))}:&=\left\| \left(\int_0^{\infty} \!  \sup_{x\in\bR^+}| \mathcal{X}(t,x) |^p  dt \right)^{1/p} \right\|_{L^p(\Omega, \sF, \bP)}<\infty.
\end{align*}
Similarly, $\mathcal L^p(C_0(\bR^+)\oplus \bR)$ 
is defined as the space of adapted functions $\mathcal{Y}: \Omega\times \bR^+\times\bR^+\to\bR$ admitting a unique decomposition
\[
\mathcal{Y}(t,x) = \mathcal{Y}_0(t,x) + \mathcal{Y}_{\infty}(t),
\]
where $\mathcal{Y}_0\in\mathcal L^p(C_0(\bR^+))$ 
and $\mathcal{Y}_{\infty}$ is 
a real-valued $(\sF_t)_{t\geq 0}$-adapted process satisfying
\begin{align*}
\| \mathcal Y \|_{\mathcal L^p(C_0(\bR^+)\oplus \bR)} &:
    = \| \mathcal Y_0 \|_{\mathcal L^p(C_0(\bR^+))} + \|\mathcal Y_{\infty}\|_{L^p(\Omega\times\bR^+,\bP\otimes dt)}<\infty.
\end{align*}
 With these norms, 
 \[
    \big(\mathcal{S}^p(C_0(\bR^+)), \| \cdot \|_{\mathcal{S}^p(C_0(\bR^+))}\big), 
    \quad \mbox{and} \quad
    \big(\mathcal{S}^p(C_0(\bR^+)\oplus \bR), \| \cdot \|_{\mathcal{S}^p(C_0(\bR^+)\oplus \bR)}\big), 
\]
as well as
\[
    \big(\mathcal{L}^p(C_0(\bR^+)), \| \cdot \|_{\mathcal{L}^p(C_0(\bR^+))}\big)
    \quad \mbox{and} \quad
     \big(\mathcal{L}^p(C_0(\bR^+)\oplus \bR), \| \cdot \|_{\mathcal{L}^p(C_0(\bR^+)\oplus \bR)}\big),
 \]
are Banach spaces. 

\section{The model}\label{sec:model}

We consider a monopolistic oil or gas producer who optimally exploits an exhaustible resource. Unlike standard extraction models, we assume that the total resource capacity $\overline X$ is random and that the producer knows its distribution but not its realization. The random variable $\overline X$ is defined on a filtered probability space $(\Omega,\sF,(\sF_t)_{t\geq0},\bP)$ satisfying the usual conditions, and expectations under $\bP$ are denoted by $\bE[\cdot]$. For simplicity, we assume that $\overline X$ is nonnegative, has an atomless law, and has unbounded support. Thus, for every $x\geq0$,
\[
    \bP_x(\cdot):=\bP(\cdot\mid x<\overline X)
\]
is well defined, with conditional expectations
\[
    \bE_{\bP_x}[\cdot\mid\sF_t]
    =
    \frac{\bE[\cdot\,1_{\{x<\overline X\}}\mid\sF_t]}{\bP(x<\overline X|\sF_t)}.
\]

The filtration $(\sF_t)_{t\geq0}$ models the arrival of information that affects the producer's assessment of the resource stock. We assume that $(\sF_t)_{t\geq0}$ is generated by a Wiener process $W$ and set
$$
    \sF_\infty:=\bigvee_{t\geq0}\sF_t \subseteq \sF.
$$
The information flow may represent geological surveys, drilling outcomes, technological developments, and other signals that affect reserve estimates and subsequent extraction decisions; see \cite{DeshmukhPliska1980,SwierzbinskiMendelsohn1989} for more details. Although the filtration is Brownian, the conditional distribution of reserves may depend on the entire observed history rather than on a finite-dimensional current state. Our formulation therefore accommodates genuinely non-Markovian control problems with random coefficients.

\begin{rmk}
For our application, it is useful to think of
$$
    \overline X=X+Y,
$$
where $X$ is independent of $\sF_\infty$ and $Y$ is $\sF_\infty$-measurable. The variable $Y$ represents the learnable component of reserves, whose uncertainty is progressively resolved through the information flow, whereas $X$ represents residual geological uncertainty that cannot be resolved in this way. Thus, information changes the producer's assessment of total reserves through the conditional distribution of $Y$, while uncertainty about $X$ remains irreducible. This distinction allows uncertainty about the resource stock to persist even in the long run.
\end{rmk}

The benchmark case in which $\overline X$ is known initially corresponds to the model studied in \cite{GraeweHorstSircar2022}. More generally, $\overline X$ need not be $\sF_\infty$-measurable. In this case, the information flow gradually reduces uncertainty about the resource capacity through the conditional survival function
\[
    \hat F(t,x):=\bP(\overline X>x\mid\sF_t),
\]
while residual uncertainty may persist even in the long run. We impose the following standing assumption.

\begin{ass}\label{assumption}
The random variable $\overline X$ has finite expectation. Moreover, the survival function is a.s.~strictly positive,
\[
    \hat F(t,x)>0
    \quad\mbox{a.s.~for all }(t,x)\in\bR^+\times\bR^+,
\]
and Lipschitz continuous in the state variable: there exists $L_{\hat F}\geq0$ such that
\[
    \esssup_{\omega\in\Omega}
    |\hat F(t,x)-\hat F(t,\overline x)|
    \leq L_{\hat F}|x-\overline x|.
\]
\end{ass}

\begin{rmk}
Since $0\leq\hat F\leq1$ and $\bE[\overline X]<\infty$, Fubini's theorem yields, for every $p>1$ and $T>0$,
\[
    \bE\left[\int_0^T\int_0^\infty \hat F^p(t,x)\,dx\,dt\right]
    \leq
    \bE\left[\int_0^T\int_0^\infty \hat F(t,x)\,dx\,dt\right]
    \leq
    T\bE[\overline X]
    <\infty.
\]
This integrability condition will be used in the uniqueness proof for viscosity solutions.
\end{rmk}

\subsection{The time-inconsistent control problem}

The monopolist extracts the resource at a progressively measurable rate $\alpha:[0,\infty)\to[0,1]$. The set of admissible extraction rates is denoted by $\cA$. Given an amount $x\geq0$ already extracted at or before time $t \geq 0$, cumulative extraction evolves as
\begin{equation}\label{stochastic-state-process}
X_s^{t,x;\alpha}
=
x+\int_t^s\alpha_u\,du,
\quad s\geq t,
\end{equation}
until the resource is depleted. The depletion time (as seen from time zero) is
\[
    \tau^{x,\alpha}
    :=
    \inf\{t\geq0:X_t^{0,x;\alpha}\geq\overline X\},
    \qquad
    \inf\varnothing:=\infty.
\]
We emphasize that $\tau^{x,\alpha}$ is a random time but need not be a stopping time.

Following \cite{ChanSircar2015,ChanSircar2017,GraeweHorstSircar2022}, we fix a discount rate $r>0$ and assume that the inverse demand function is $p(a)=1-a$ for production rates $a\in[0,1]$. The monopolist's discounted revenue is therefore
\[
    J(0,x,\alpha)
    =
    \int_0^{\tau^{x,\alpha}}e^{-rt}\alpha_t(1-\alpha_t)\,dt
    =
    \int_0^\infty e^{-rt}
    {\bf 1}_{\{\tau^{x,\alpha}>t\}}
    \alpha_t(1-\alpha_t)\,dt.
\]
When computing expected revenues we need to account for the distinct learning mechanisms. For any initial state $x$, the monopolist knows that $\overline X>x$ and therefore updates her belief about total reserves to $\bP_x$. Moreover, she has access to the exogenous information $\sF_0$. The value function at time $t=0$ is hence given by,
\begin{align*}
V(0,x)
&:=
\esssup_{\alpha\in\cA}
\bE_{\bP_x}\left[
\int_0^{\tau^{x,\alpha}}
e^{-rt}\alpha_t(1-\alpha_t)\,dt
\Big|\sF_0
\right]
\\
&=
\esssup_{\alpha\in\cA}
\bE\left[
\int_0^\infty
e^{-rt}
\frac{\alpha_t(1-\alpha_t)}{\hat F(0,x)}
1_{\{t<\tau^{x,\alpha}\}}\,dt
\Big|\sF_0
\right]
\\
&=
\esssup_{\alpha\in\cA}
\bE\left[
\int_0^\infty
\frac{\hat F(t,X_t^{0,x;\alpha})}{\hat F(0,x)}
e^{-rt}\alpha_t(1-\alpha_t)\,dt
\Big|\sF_0
\right],
\end{align*}
where the second equality follows from the adaptedness of the extraction strategy and Fubini's theorem for conditional expectations. The term $\hat F(t,X_t^{0,x;\alpha})$ captures both forms of learning: exogenous learning through the information contained in $\sF_t$ and endogenous learning through the fact that the resource has not yet been depleted at the cumulative extraction level $X_t^{0,x;\alpha}$. For a general initial condition $(t,x)$, conditional on the resource not yet being depleted, the monopolist evaluates
\begin{align}
J(t,x;\alpha)
=
\bE\left[
\int_t^\infty
\frac{\hat F(s,X_s^{t,x;\alpha})}{\hat F(t,x)}
e^{-rs}\alpha_s(1-\alpha_s)\,ds
\Big|\sF_t
\right],
\label{dynamic-functional_ti}
\end{align}
and the corresponding stochastic value function is
\begin{align}
V(t,x)
=
\esssup_{\alpha\in\cA}J(t,x;\alpha).
\label{dynamic-value-functional_ti}
\end{align}
Because the normalizing factor $\hat F(t,x)$ changes with the current state and information set, the family of continuation problems does not satisfy the standard additive recursion underlying dynamic programming. The control problem is therefore time-inconsistent. 



The following examples contrast a finite-dimensional learning model with a genuinely path-dependent one.

\begin{example}\label{example:conditional-lognormal}
Let
\[
Y
=
\exp\left(
\log m
-\frac12\int_0^\infty\overline\sigma_s^2\,\ud s
+\int_0^\infty\overline\sigma_s\,\ud W_s
\right),
\qquad m>0,
\]
where $\overline\sigma:[0,\infty)\to\bR$ is deterministic and satisfies
\[
\int_0^\infty\overline\sigma_s^2\,\ud s<\infty,
\qquad
v_t:=\int_t^\infty\overline\sigma_s^2\,\ud s>0,
\quad t\geq0.
\]
Let $E\sim\mathrm{Exp}(1)$ be independent of $\sF_\infty$ and set
$\overline X:=Y+E$. With
\[
\mu_t
:=
\log m
-\frac12\int_0^\infty\overline\sigma_s^2\,\ud s
+\int_0^t\overline\sigma_s\,\ud W_s,
\]
the conditional distribution of $\log Y$ given $\sF_t$ is normal with mean $\mu_t$ and variance $v_t$. Hence
\[
\hat F(t,x)
=
\bE[g_x(Y)\mid\sF_t],
\qquad
g_x(y)
:=
\begin{cases}
1, & x\leq y,\\[1mm]
e^{-(x-y)}, & x>y.
\end{cases}
\]
This gives a finite-dimensional learning model in which the conditional law of the learnable component $Y$ is determined by $(\mu_t,v_t)$. Moreover, since $x\mapsto g_x(y)$ is globally $1$-Lipschitz uniformly in $y$, conditional expectation yields
\[
|\hat F(t,x)-\hat F(t,\bar x)|\leq |x-\bar x|,
\]
so Assumption~\ref{assumption} holds with $L_{\hat F}=1$.
\end{example}

\begin{example}\label{example:path-dependent-learning}
Fix $H\in(1/2,1)$ and let
\[
B_t^H
:=
\int_0^t K_H(t,s)\,\ud W_s,
\quad t\geq0,
\]
be a fractional Brownian motion represented through $W$, where $K_H$
denotes the corresponding Volterra kernel. Define
\[
I_\infty
:=
\int_0^\infty e^{-s}\tanh(B_s^H)\,\ud s,
\qquad
\Theta
:=
\frac12+\frac14 I_\infty. 
\]
Since $|I_\infty|\leq1$, we have $\Theta\in[1/4,3/4]$. Let $E\sim\mathrm{Exp}(1)$ be independent of $\sF_\infty$ and set
\[
\overline X:=\Theta+E.
\]
Then, with $g_x$ as in the preceding example,
\[
\hat F(t,x)
=
\bE[g_x(\Theta)\mid\sF_t].
\]
Thus, $\Theta$ represents a learnable component of reserves, while the
independent random variable $E$ represents residual uncertainty that
cannot be resolved through the information flow. Moreover, the learning mechanism is genuinely path dependent. Indeed, for
$\ell>t$,
\[
\bE[B_\ell^H\mid\sF_t]
=
\int_0^t K_H(\ell,s)\,\ud W_s,
\]
which depends on the observed path $W_{\cdot\wedge t}$ rather than only
on the current value $B_t^H$. Consequently, the conditional distribution
of $\Theta$, and hence the survival function $\hat F(t,x)$, may depend on
the entire history of the observed signal. The example combines persistent reserve uncertainty with genuinely non-Markovian learning. As above, $x\mapsto g_x(\theta)$ is uniformly $1$-Lipschitz, so Assumption~\ref{assumption} again holds with $L_{\hat F}=1$.
\end{example}

\subsection{The time-consistent control problem}

The time-inconsistent problem admits an equivalent time-consistent reformulation with reward functional
\begin{align*}
\tilde J(t,x;\tilde\alpha)
&=
\bE\left[
\int_t^\infty
\hat F(s,X_s^{t,x;\tilde\alpha})
e^{-rs}\tilde\alpha_s(1-\tilde\alpha_s)\,ds
\Big|\sF_t
\right],
\end{align*}
where the conditional expectation is taken under the benchmark probability measure $\bP$, and value function
\begin{align}
\tilde V(t,x)
&=
\esssup_{\tilde\alpha\in\cA}
\tilde J(t,x;\tilde\alpha).
\label{dynamic-value-functional}
\end{align}
Our discounting convention implies
\[
\lim_{t\to\infty}\tilde V(t,x)=0.
\]
Although nonstandard, discounting from calendar time zero rather than from the current time does not change the nature of the control problem and conveniently yields a boundary condition at temporal infinity. By Assumption~\ref{assumption}, $\tilde V$ is bounded and, since $\hat F(t,\cdot)$ is non-increasing, $\tilde V(t,\cdot)$ is a.s.~non-increasing as well.

The next proposition collects basic properties of the value function. Its proof is analogous to that of \cite[Proposition 3.3]{qiu_stochastic_hjb} and is omitted.

\begin{prop}\label{proposition}
\begin{enumerate}
\item[(i)] For every $\epsilon>0$ and $(t,x)\in\bR^+\times\bR^+$, there exists $\tilde\alpha\in\cA$ such that
\[
\bE\left[
\tilde V(t,x)-\tilde J(t,x;\tilde\alpha)
\right]
<\epsilon.
\]

\item[(ii)] With probability one, $\tilde V$ and, for every $\alpha\in\cA$, $\tilde J(\cdot,\cdot;\alpha)$ are continuous on $\bR^+\times\bR^+$.

\item[(iii)] For every $(\alpha,x)\in\cA\times\bR^+$,
\[
\{\tilde V(s,X_s^{0,x;\alpha})\}_{s\in\bR^+}
\]
is a continuous process.

\item[(iv)] There exists $L_V>0$ such that, for every $(\alpha,t)\in\cA\times\bR^+$,
\[
|\tilde V(t,x)-\tilde V(t,y)|
+
|\tilde J(t,x;\alpha)-\tilde J(t,y;\alpha)|
\leq
L_V|x-y|
\quad\text{a.s.},
\qquad
x,y\in\bR^+.
\]
\end{enumerate}
\end{prop}

By construction,
\begin{equation}\label{scale_feature}
\tilde V(t,x)
=
\hat F(t,x)V(t,x).
\end{equation}
Hence it is sufficient to analyze the time-consistent problem. To this end, define the random Hamiltonian
\[
\cH(t,x,p)
:=
\max_{a\in[0,1]}
\left\{
ap+e^{-rt}a(1-a)\hat F(t,x)
\right\}.
\]
We show below that $\tilde V$ is an adapted random field and the unique viscosity solution, in the sense of Definition~\ref{def_vis_sol}, of the stochastic Hamilton--Jacobi equation
\begin{equation}\label{SHJB_lim0}
-\mathfrak d_t\tilde V(t,x)
=
\cH(t,x,\partial_x\tilde V(t,x)),
\qquad
\lim_{t\to\infty}
\sup_{x\in\bR^+}
|\tilde V(t,x)|
=
0
\quad\text{a.s.},
\end{equation}
for $(t,x)\in\bR^+\times\bR^+$. If the viscosity solution is sufficiently regular, the optimal time-consistent strategy is
\begin{equation*}
\tilde\alpha^*(t,x)
=
\frac{
\left(
\partial_x\tilde V(t,x)
+
e^{-rt}\hat F(t,x)
\right)^+
}{
2e^{-rt}\hat F(t,x)
}.
\end{equation*}

\begin{rmk}
In the notation used for stochastic viscosity solutions, an adapted random field $\tilde V$ admits the semimartingale decomposition
\[
d\tilde V(t,x)
=
\mathfrak d_t\tilde V(t,x)\,dt
+
\mathfrak d_\omega\tilde V(t,x)\,dW_t;
\]
see Section~\ref{section_sto_con} for the definitions of $\mathfrak d_t\tilde V$ and $\mathfrak d_\omega\tilde V$. Comparing this decomposition with the stochastic Hamilton--Jacobi equation gives
\[
-\mathfrak d_t\tilde V(t,x)
=
\cH\bigl(t,x,\partial_x\tilde V(t,x)\bigr).
\]
The martingale component $\mathfrak d_\omega\tilde V(t,x)$ is determined by martingale representation. In this sense, the stochastic HJ equation can be written compactly in terms of the drift operator $\mathfrak d_t$ as in \eqref{SHJB_lim0}.
In view of \eqref{scale_feature}, the scaled time-inconsistent value function satisfies
\begin{equation*}
-\mathfrak d_t[\hat F(t,x)V(t,x)]
=
\frac{
\left|
\left(
\partial_x[\hat F(t,x)V(t,x)]
+
e^{-rt}\hat F(t,x)
\right)^+
\right|^2
}{
4e^{-rt}\hat F(t,x)
},
\qquad
\lim_{t\to\infty}
\hat F(t,x)V(t,x)
=
0.
\end{equation*}
Assuming sufficient regularity, the optimal strategies of the time-inconsistent problem and its time-consistent counterpart coincide:
\begin{equation*}
\alpha^*(t,x)
=
\frac{
\left(
\partial_x[\hat F(t,x)V(t,x)]
+
e^{-rt}\hat F(t,x)
\right)^+
}{
2e^{-rt}\hat F(t,x)
}
=
\tilde\alpha^*(t,x).
\end{equation*}
\end{rmk}

\section{Viscosity solutions for the time-consistent problem}\label{section_sto_con}
We characterize the value function as the unique viscosity solution of the stochastic HJ equation. In the stochastic setting, test functions are semimartingale random fields and therefore carry a martingale component in addition to their spatial derivatives. This motivates the following class of test processes.

\begin{defn}\label{cf1}
For $\phi\in\mathcal{S}^2(C_0(\bR^+)\oplus \bR)$ with $\partial_x\phi\in\mathcal{L}^2(C_0(\bR^+))$, 
we say that $\phi\in\mathcal{C}_{\sF}$ if there exists
\[
(\mathfrak{d}_t\phi,\mathfrak{d}_{\omega}\phi)
\in\mathcal{L}^2(C_0(\bR^+)\oplus \bR)
\times\mathcal L^{2}(C_0(\bR^+)\oplus \bR)
\]
such that: 
\begin{enumerate}
    \item [(i)]
For all $0\le \tilde r\le \tau< \infty$ and $x\in\bR^+$,
\begin{equation*}
\phi(\tau,x)=\phi(\tilde r,x)+\int_{\tilde r}^{\tau} \! \mathfrak{d}_s \phi(s, x)\,\ud s +\int_{\tilde r}^{\tau} \! \mathfrak{d}_{\omega} \phi(s, x)\,\ud W_s, a.s.;
\end{equation*}
\item [(ii)] For each $N,K>0$, there exist $p>1$, $q\in[1,\infty)$, $\beta\in \big(\frac{1}{p},1\big)$, and an $(\sF_t)_{t\ge0}$-adapted process $L_{\phi}$ such that, a.s., for almost every $t\in(0,\infty)$ 
and all $x, \overline x\in[0,K]$,
 \begin{align*}
 | \mathfrak{d}_t \phi(t, x)-\mathfrak{d}_t \phi(t, \overline x) | + | \partial_x \phi(t, x)-\partial_x \phi(t, \overline x) | 
 \le L_{\phi}(\omega,t) (1+ |x|^q + |\overline x|^q)| x-\overline x |^{\beta},
 \end{align*}
 where $L_{\phi}\in L^p(\Omega\times[0,N);\bR^+)$.
\end{enumerate}
\end{defn}

 \begin{rmk}
Each $\phi\in\mathcal C_{\sF}$ is an It\^o semimartingale parametrized by $x\in\bR^+$. Uniqueness of the semimartingale decomposition yields uniqueness of $(\mathfrak d_t\phi,\mathfrak d_\omega\phi)$ in $\mathcal L^2(C_0(\bR^+)\oplus\bR)\times\mathcal L^2(C_0(\bR^+)\oplus\bR)$, so the operators $\mathfrak d_t$ and $\mathfrak d_\omega$ are well defined; cf.~\cite[Section 5.2]{cont2013functional} and \cite[Theorem 4.3]{leao2018weak}. For deterministic $\phi$, $\mathfrak d_\omega\phi=0$ and $\mathfrak d_t\phi=\partial_t\phi$. Under suitable Malliavin regularity, $\mathfrak d_\omega\phi(t,x)$ can be identified with the diagonal Malliavin derivative of $\phi(t,x)$.      
 \end{rmk}

We will use the following It\^o--Kunita formula; its proof is analogous to \cite[Pages 118--119]{kunita1981some}.

\begin{lem}[It\^o--Kunita formula]\label{Ito-Kunita}
    Suppose $u\in\mathcal C_{\sF}$. Then, for each $\alpha\in\cA$, almost surely, for every $0\le\rho\le\tau<\infty$ and $x\in\bR^+$, 
    \begin{equation*}
        u(\tau,X_{\tau}^{\rho,x;\alpha})-u(\rho,x) = \int_{\rho}^{\tau}\! \left[ \mathfrak{d}_s u(s,X_s^{\rho,x;\alpha})+\alpha_s \partial_x u(s,X_s^{\rho,x;\alpha}) \right]\,\ud s + \int_{\rho}^{\tau}\!  \mathfrak{d}_{\omega} u(s,X_{s}^{\rho,x;\alpha})\,\ud W_s.
    \end{equation*}
\end{lem}

\subsection{Viscosity solutions}

For $t\ge0$, let $\mathcal T^t$ denote the stopping times with values in $[t,\infty)$ and $\mathcal T^t_+$ those satisfying $\tau>t$. These classes are used to formulate the optimal-stopping notion of tangency. For $(u,\tau)\in\mathcal S^2(C_0(\bR^+))\times\mathcal T^0$, $\Omega_\tau\in\sF_\tau$ with $\bP(\Omega_\tau)>0$, and $\xi\in L^0(\Omega,\sF_\tau;\bR^+)$, define:

\begin{align*}
\underline{\mathcal{G}}u(\tau, \xi; \Omega_{\tau}):=&\Bigg\{ \phi\in\mathcal{C}_{\sF}: (\phi-u)(\tau, \xi)1_{\Omega_{\tau}} = 0
\\
&=\essinf_{\overline{\tau}\in \mathcal{T}^{\tau}} \bE_{\sF_{\tau}} \left[ \inf_{y\in\bR^+}(\phi-u)(\overline{\tau}\land \hat{\tau},y) \right]1_{\Omega_{\tau}} \text{ }a.s.\text{ for some } \hat\tau\in \mathcal T_+^{\tau}\Bigg\},
\end{align*}

\begin{align*}
\overline{\mathcal{G}}u(\tau, \xi; \Omega_{\tau}):=&\Bigg\{ \phi\in\mathcal{C}_{\sF}: (\phi-u)(\tau, \xi)1_{\Omega_{\tau}} = 0
\\
&=\esssup_{\overline{\tau}\in \mathcal{T}^{\tau}} \bE_{\sF_{\tau}} \left[ \sup_{y\in\bR^+}(\phi-u)(\overline{\tau}\land \hat{\tau},y) \right]1_{\Omega_{\tau}} \text{ }a.s.\text{ for some } \hat\tau\in \mathcal T_+^{\tau}\Bigg\}.
\end{align*}

The test functions are adapted semimartingale random fields, and tangency is imposed in conditional expectation rather than pathwise. If
\(\phi\in\underline{\mathcal G}u(\tau,\xi;\Omega_\tau)\), then \((\phi-u)(\tau,\xi)=0\) on \(\Omega_\tau\), while
\[
0=
\essinf_{\bar\tau\in\mathcal T^\tau}
\bE_{\sF_\tau}
\left[
\inf_y(\phi-u)(\bar\tau\wedge\hat\tau,y)
\right]
\quad\text{on }\Omega_\tau.
\]
is the stochastic analogue of a local minimum. The stopping time \(\hat\tau>\tau\) localizes the condition, and the essential infimum over future stopping times replaces pathwise minimization in time. Equivalently, with \(Z_s:=\inf_y(\phi-u)(s,y)\),
\[
\essinf_{\bar\tau\in\mathcal T^\tau}
\bE_{\sF_\tau}[Z_{\bar\tau\wedge\hat\tau}]
=Z_\tau=0.
\]
Thus no admissible stopping rule can produce a negative conditional expected gap locally after \(\tau\).

\begin{defn}\label{def_vis_sol}
    We say that $u\in\mathcal S^2(C_0(\bR^+))$ is a viscosity subsolution (resp. supersolution) of the stochastic Hamilton--Jacobi equation \eqref{SHJB_lim0} if
    \[
    \lim_{T\to\infty}\sup_{y\in\bR^+}u(T,y)\le0
    \qquad
    \left(\text{resp., }\lim_{T\to\infty}\inf_{y\in\bR^+}u(T,y)\ge0\right)
    \quad\text{a.s.},
    \]
    and if, for any $\tau\in\mathcal T^0$, $\Omega_{\tau}\in\sF_{\tau}$ with $\bP(\Omega_{\tau})>0$, $\xi\in L^0(\Omega,\sF_{\tau};\bR^+)$, and any $\phi\in\underline{\mathcal G}u(\tau,\xi;\Omega_{\tau})$ (resp. $\phi\in\overline{\mathcal G}u(\tau,\xi;\Omega_{\tau})$), it holds that
    \begin{equation}\label{subsol}
    \operatorname{ess}\liminf_{(s,z)\to(\tau^{+},\xi)}
    \left\{-\mathfrak{d}_s\phi(s,z)-\mathcal H(s,z,\partial_z\phi(s,z))\right\}
    \le0,
    \quad\text{for a.e. }\omega\in\Omega_{\tau},
    \end{equation}
    \begin{equation}\label{supsol}
    \left(\text{resp., }\operatorname{ess}\limsup_{(s,z)\to(\tau^{+},\xi)}
    \left\{-\mathfrak{d}_s\phi(s,z)-\mathcal H(s,z,\partial_z\phi(s,z))\right\}
    \ge0,
    \quad\text{for a.e. }\omega\in\Omega_{\tau}\right).
    \end{equation}
    The function $u$ is a viscosity solution to the stochastic HJ equation \eqref{SHJB_lim0} if it is both a viscosity subsolution and a viscosity supersolution.
\end{defn}

\subsection{Existence of viscosity solutions}
The existence proof relies on the following dynamic programming principle, whose proof is analogous to \cite[Theorem 3.4]{qiu_stochastic_hjb}.

\begin{thm}[Dynamic programming principle]\label{DPP}
    Let Assumption~\ref{assumption} hold. For any stopping time $\tau,\hat{\tau}$, with $0\le\tau\le\hat{\tau}<\infty$, and any $\xi\in L^0(\Omega,\sF_{\tau};\bR^+)$, we have
    \begin{equation*}
        \tilde V(\tau,\xi)=\esssup_{\alpha\in\cA} \bE_{\sF_{\tau}}\left[ \int_{\tau}^{\hat{\tau}}\! e^{-rs}\alpha_s(1-\alpha_s)\hat F(s,X_s^{\tau,\xi;\alpha}) ds + \tilde V(\hat{\tau},X_{\hat{\tau}}^{\tau,\xi;\alpha}) \right].
    \end{equation*}
\end{thm}

The DPP yields the following almost-sure near-optimality estimate, using stability of $\mathcal A$ under finite and countable $\sF_\tau$-measurable pasting. The proof is standard.

\begin{lem}\label{conditional_h_optimal} 
Let $\tau$ and $ \theta$ be stopping times with $0\leq \tau \leq \theta<\infty$ a.s., and let
$\xi\in L^0(\Omega,\sF_\tau;\bR^+)$. For each $\alpha\in\cA$, define
\[
Y_\alpha
=
\bE_{\sF_\tau}
\left[
\int_\tau^\theta
e^{-rs}\alpha_s(1-\alpha_s)
\hat F(s,X_s^{\tau,\xi;\alpha})\,\ud s
+
\tilde V(\theta,X_\theta^{\tau,\xi;\alpha})
\right].
\]
Then, for every $\eta>0$, there exists $\alpha^\eta\in\cA$ such that
\[
Y_{\alpha^\eta}
\geq
\esssup_{\alpha\in\cA}Y_\alpha-\eta,
\quad \text{a.s.}
\]
Consequently, by the dynamic programming principle,
\begin{equation*}
\tilde V(\tau,\xi)
\leq
\bE_{\sF_\tau}
\left[
\int_\tau^\theta
e^{-rs}\alpha^\eta_s(1-\alpha^\eta_s)
\hat F(s,X_s^{\tau,\xi;\alpha^\eta})\,\ud s
+
\tilde V(\theta,X_\theta^{\tau,\xi;\alpha^\eta})
\right]
+\eta,
\quad \text{a.s.}
\end{equation*}
\end{lem}

We can now prove the viscosity property; notably, no boundedness estimate for the nonlinear Hamiltonian $\mathcal H$ is needed.

\begin{thm}\label{vis_existence}
    Under Assumption~\ref{assumption}, the value function \eqref{dynamic-value-functional} is a viscosity solution to the stochastic Hamilton--Jacobi equation \eqref{SHJB_lim0}.
    \begin{proof}
        Since $\lim_{t\to\infty}\tilde V(t,x)=0$ for all $x\in\bR^+$ a.s.~and $\tilde V\in\mathcal S^p(C_0(\bR^+))$ for every $p>1$, it remains to verify \eqref{subsol} and \eqref{supsol}.

\medskip
        \textbf{Step 1. Supersolution property.} Suppose, to the contrary, that \eqref{supsol} fails. Then there exists
        \[\phi\in\overline{\mathcal G}\tilde V(\tau,\xi;\Omega_\tau)\] 
        with $\tau\in\mathcal T^0,\Omega_{\tau}\in\sF_{\tau},\bP(\Omega_{\tau})>0$, $\xi\in L^0(\Omega,\sF_{\tau};\bR^+)$ and associated stopping time $\hat \tau$ satisfying $\hat \tau>\tau$ on $\Omega_\tau$, 
        such that there exist $\epsilon,\delta\in (0,1)$, 
        and $\Omega'\in\sF_{\tau}$ with $\Omega'\subset\Omega_{\tau}$, $\Omega'\subset\{\tau<\hat\tau\}$, and $\bP(\Omega')>0$ such that, a.s.~on $\Omega'$,
        \begin{equation*}
            \esssup_{s\in(\tau,(\tau+\delta^2)\land\hat\tau],z\in\bR^+,|z-\xi|<\delta} \left\{ -\mathfrak{d}_s\phi(s,z)-\mathcal H(s,z,\partial_z\phi(s,z)) \right\} \le - 2\epsilon.
        \end{equation*}
        
        Without loss of generality, let $\Omega_{\tau}=\Omega'=\Omega$, $\bP(\hat \tau>\tau + \delta^2)=1$, $\bP(\tau<1)=1$, and $\xi \in L^{\infty}(\Omega,\sF_{\tau};\bR^+)$. 
        By the measurable selection theorem, there exists $\bar \alpha\in\cA$ such that, a.s.,
        \begin{equation*}
            -\mathfrak{d}_s\phi(s,\xi)- \bar\alpha_s\partial_x\phi(s,\xi) -e^{-rs}\bar\alpha_s(1-\bar\alpha_s)\hat F(s,\xi) \le -\epsilon,
        \end{equation*}
        for $\bP\otimes ds$-a.e.~$(\omega,s)$ with $\tau<s\le\tau+\delta^2$.  Let us now choose $h\in (0,\delta^2)$. 
        By the definition of $\overline{\mathcal G}\tilde V(\tau,\xi;\Omega_{\tau})$, Lemma~\ref{Ito-Kunita} and Theorem~\ref{DPP}, we have that
        \begin{align*}
            0 = &\frac{\tilde V(\tau,\xi)-\phi(\tau,\xi)}{h} 
            \\
            = &\frac{1}{h} \left( \esssup_{\alpha\in\cA}\bE_{\sF_{\tau}}\left[ \int_{\tau}^{\tau+h }\! e^{-rs}\alpha_s(1-\alpha_s)\hat F(s,X_s^{\tau,\xi;\alpha}) ds 
            + \tilde V(\tau+h,X_{\tau+h }^{\tau,\xi;\alpha}) \right] - \phi(\tau,\xi) \right)
            \\
            = &\frac{1}{h} \esssup_{\alpha\in\cA}\bE_{\sF_{\tau}}\left[ \int_{\tau}^{\tau+h }\! e^{-rs}\alpha_s(1-\alpha_s)\hat F(s,X_s^{\tau,\xi;\alpha}) ds + \tilde V((\tau+h)\land\hat \tau,X_{(\tau+h)\land\hat \tau}^{\tau,\xi;\alpha}) - \phi(\tau,\xi) \right]
            \\
            \ge &\frac{1}{h} \esssup_{\alpha\in\cA}\bE_{\sF_{\tau}}\left[ \int_{\tau}^{\tau+h}\! e^{-rs}\alpha_s(1-\alpha_s)\hat F(s,X_s^{\tau,\xi;\alpha}) ds +\phi(\tau+h,X_{(\tau+h)\land\hat \tau}^{\tau,\xi;\alpha}) - \phi(\tau,\xi) \right]
            \\
            = &\frac{1}{h} \esssup_{\alpha\in\cA}\bE_{\sF_{\tau}}\Bigg[ \int_{\tau}^{\tau+h}\! \left[e^{-rs}\alpha_s(1-\alpha_s)\hat F + \mathfrak{d}_s\phi+\partial_x\phi\alpha_s \right](s,\xi) 
            \\
            &- \big[e^{-rs}\alpha_s(1-\alpha_s)(\hat F(s,\xi) - \hat F(s,X_s^{\tau,\xi;\alpha})) + \mathfrak{d}_s\phi(s,\xi) - \mathfrak{d}_s\phi(s,X_s^{\tau,\xi;\alpha})+\alpha_s(\partial_x\phi(s,\xi) - \partial_x\phi(s,X_s^{\tau,\xi;\alpha}) ) \big]   ds \Bigg]
            \\
            \ge 
            &\frac{1}{h} \bE_{\sF_{\tau}} \left[ \int_{\tau}^{\tau+h}\!\left( \mathfrak{d}_s\phi + \bar\alpha_s\partial_x\phi  +e^{-rs}\bar\alpha_s(1-\bar\alpha_s)\hat F  \right)(s,\xi)ds \right]
            \\
            &- \frac{1}{h}  \bE_{\sF_{\tau}} \left[ \int_{\tau}^{\tau+h} L_{\hat F} h 
            + |\mathfrak{d}_s\phi(s,X_s^{\tau,\xi;\bar\alpha})-\mathfrak{d}_s\phi(s,\xi)| + |\partial_x\phi(s,\xi)-\partial_x\phi(s,X_s^{\tau,\xi;\bar\alpha})|ds\right]
            \\
            \ge &\epsilon - L_{\hat F}h - \frac{1}{h} 
            \bE_{\sF_{\tau}} \left[ \int_{\tau}^{\tau+h } 
            \Big|\mathfrak{d}_s\phi(s,\xi)-\mathfrak{d}_s\phi\Big(s,\xi+\int_{\tau}^{s}\bar \alpha_t dt\Big)\Big| 
                + \Big|\partial_x\phi(s,\xi)-\partial_x\phi\Big(s,\xi+\int_{\tau}^{s}\bar\alpha_t dt\Big)\Big| ds \right]
            \\
            \geq 
            &\epsilon - L_{\hat F}h - \frac{1}{h} 
            \bE_{\sF_{\tau}} \left[ 
                \int_{\tau}^{\tau+h } 
                L_{\phi} (s) (1+|\xi|^q+|\xi + h|^q) |h|^{\beta} \, ds \right]
            \\
            \geq &
            \epsilon - L_{\hat F}h -  3(|1+\xi|^q) h^{\beta -\frac{1}{p}}  
            \bE_{\sF_{\tau}} \left[ 
                 \int_{0}^{1+\delta^2 } 
                |L_{\phi} (s)|^p   ds
               \right]^{1/p}
            \\
            \to &\epsilon, \text{ as }h\to 0^+,
        \end{align*}
        where the constants $p>1,q\geq 1,\beta>1/p$, and the $p$-integrable function $L_{\phi}$ depend on $\phi\in \mathcal{C}_{\sF}$ through Definition~\ref{cf1}.   As $\epsilon>0$, this leads to a contradiction.

        \medskip
        \textbf{Step 2. Subsolution property.} Suppose instead that \eqref{subsol} fails. Then there exists
        \[\phi\in\underline{\mathcal G}\tilde V(\tau,\xi;\Omega_\tau)\]
         with $\tau\in\mathcal T^0,\Omega_{\tau}\in\sF_{\tau},\bP(\Omega_{\tau})>0$ and $\xi\in L^0(\Omega,\sF_{\tau};\bR^+)$ 
         such that there exist $\epsilon,\delta\in(0,1)$ and $\Omega'\in\sF_{\tau}$ with $\Omega'\subset\Omega_{\tau}$, $\Omega'\subset\{\tau<\hat\tau\}$, and $\bP(\Omega')>0$ such that, a.s.~on $\Omega'$, 
        \begin{equation}\label{sub_opposite}
            \essinf_{s\in(\tau,(\tau+\delta^2)\land\hat \tau],z\in\bR^+,|z-\xi|<\delta} \left\{ -\mathfrak{d}_s\phi(s,z)-\mathcal H(s,z,\partial_z\phi(s,z)) \right\} \ge \epsilon.
        \end{equation}
        Without loss of generality,
         let $\Omega_{\tau}=\Omega'=\Omega$, $\bP(\tau<1,\hat \tau>\tau +\delta^2)=1$, and $\xi \in L^{\infty}(\Omega,\sF_{\tau};\bR^+)$. 
         Set $\tau^h=\tau+h$. 
         For each $h\in (0,\delta^2)$, we put $\eta=h^2$ and apply Lemma~\ref{conditional_h_optimal} to obtain a near-optimal control $\alpha^h\in\cA$ such that 
         \begin{align*}
         \tilde V(\tau,\xi) \leq \bE_{\sF_{\tau}} \left[ \int_{\tau}^{\tau^h}\! e^{-rs}\alpha_s^h(1-\alpha_s^h)\hat F(s,X_s^{\tau,\xi;\alpha^h}) ds + \tilde V(\tau^h,X_{\tau^h}^{\tau,\xi;\alpha^h}) \right] + h^2.
         \end{align*}
         As $|X_s^{\tau,\xi;\alpha^h}-\xi| \leq |h|\leq \delta$,   we have by \eqref{sub_opposite} that  
        \begin{equation*} 
            -\mathfrak{d}_s\phi(s,X_s^{\tau,\xi;\alpha^h})-\alpha_s^h\partial_x\phi(s,X_s^{\tau,\xi;\alpha^h}) -e^{-rs}\alpha_s^h(1-\alpha_s^h)\hat F(s,X_s^{\tau,\xi;\alpha^h}) \ge \epsilon,
        \end{equation*}
        for $\bP\otimes ds$-a.e.~$(\omega,s)$ with $\tau<s\le\tau+\delta^2$.
         Thus, the dynamic programming principle of Theorem~\ref{DPP}, the It\^o--Kunita formula of Lemma~\ref{Ito-Kunita}, and Lemma~\ref{conditional_h_optimal},
         yield the following contradiction:  
        \begin{align*}
            0 \le &\liminf_{h\to 0^+} \frac{1}{h} \bE_{\sF_{\tau}} \left[ (\phi-\tilde V)\left( \tau^h,X_{\tau^h}^{\tau,\xi;\alpha^h} \right) - (\phi-\tilde V)(\tau,\xi) \right]
            \\
            \le &\liminf_{h\to 0^+}\frac{1}{h} \bE_{\sF_{\tau}}\left[ \phi(\tau^h,X_{\tau^h}^{\tau,\xi;\alpha^h}) - \phi(\tau,\xi) + \int_{\tau}^{\tau^h}\! e^{-rs}\alpha^h_s(1-\alpha^h_s)\hat F(s,X_s^{\tau,\xi;\alpha^h})\,\ud s + h^2 \right]
            \\
            = 
            &\liminf_{h\to 0^+}\frac{1}{h} \bE_{\sF_{\tau}} \left[\int_{\tau}^{\tau^h}\! \mathfrak{d}_s\phi(s,X_s^{\tau,\xi;\alpha^h})+\alpha^h_s \partial_x\phi(s,X_s^{\tau,\xi;\alpha^h})+e^{-rs}\alpha^h_s(1-\alpha^h_s)\hat F(s,X_s^{\tau,\xi;\alpha^h}) ds + h^2 \right]
            \\
            \le &\liminf_{h\to 0^+}\frac{1}{h} \bE_{\sF_{\tau}} \Bigg[\int_{\tau}^{\tau^h}\! 
            (-\epsilon) ds +h^2 \Bigg]
            = -\epsilon < 0.
        \end{align*}
    \end{proof}
\end{thm}
 
\subsection{Comparison and uniqueness of viscosity solutions}\label{sec:uniqueness}
 
To establish the uniqueness of the viscosity solution to the stochastic HJ equation \eqref{SHJB_lim0},
we now introduce a suitable comparison principle.  
Following \cite{qiu_stochastic_hjb,qiu_uniqueness_2019}, we use a stochastic strict-contact argument based on the Snell envelope of an exponentially perturbed difference between the viscosity solution and a test function. Our proof requires neither the strong regularity assumptions of \cite{qiu_stochastic_hjb} nor the penalization approximations of \cite{qiu_uniqueness_2019}. The infinite horizon and unbounded state space create an additional localization problem. The one-sided conditions \eqref{condition:infty-sub}, \eqref{additional_condition_cp-sub}, \eqref{condition:infty}, and \eqref{additional_condition_cp} prevent violations of the desired ordering from escaping to temporal or spatial infinity and ensure attainment of the relevant extrema.

\begin{thm}[Comparison principle]\label{comparison}
Let Assumption~\ref{assumption} hold, and let $\phi\in\mathcal C_{\sF}$.
\begin{enumerate}
\item[(i)] Suppose that $u$ is a viscosity subsolution of \eqref{SHJB_lim0} and that, outside a common $\bP$-null set,
\begin{align}
\liminf_{T\to\infty}\inf_{x\in\bR^+}
    \bigl(\phi(T,x)-u(T,x)\bigr)&\ge0,
    \label{condition:infty-sub}\\
\liminf_{x\to\infty}
    \bigl(\phi(t,x)-u(t,x)\bigr)&\ge0,
    \qquad t\in\bR^+,
    \label{additional_condition_cp-sub}\\
\operatorname{ess}\liminf_{(s,x)\to(t^+,y)}
    \left\{-\mathfrak d_s\phi(s,x)
    -\mathcal H(s,x,\partial_x\phi(s,x))\right\}&\ge0,
    \qquad (t,y)\in\bR^+\times\bR^+.
\end{align}
Then, almost surely,
\[
u(t,x)\le\phi(t,x),
\qquad (t,x)\in\bR^+\times\bR^+.
\]

\item[(ii)] Suppose that $u$ is a viscosity supersolution of \eqref{SHJB_lim0} and that, outside a common $\bP$-null set,
\begin{align}
\limsup_{T\to\infty}\sup_{x\in\bR^+}
    \bigl(\phi(T,x)-u(T,x)\bigr)&\le0,
    \label{condition:infty}\\
\limsup_{x\to\infty}
    \bigl(\phi(t,x)-u(t,x)\bigr)&\le0,
    \qquad t\in\bR^+,
    \label{additional_condition_cp}\\
\operatorname{ess}\limsup_{(s,x)\to(t^+,y)}
    \left\{-\mathfrak d_s\phi(s,x)
    -\mathcal H(s,x,\partial_x\phi(s,x))\right\}&\le0,
    \qquad (t,y)\in\bR^+\times\bR^+. \label{supersolution_condition_sup-clasical}
\end{align}
Then, almost surely,
\[
u(t,x)\ge\phi(t,x),
\qquad (t,x)\in\bR^+\times\bR^+.
\]
\end{enumerate}
\end{thm}

\begin{proof}
We prove only (ii), since (i) is analogous. Suppose, to the contrary, that some $(t_0,x_0)\in\bR^+\times\bR^+$ satisfies
\[
\bP\bigl(\phi(t_0,x_0)>u(t_0,x_0)\bigr)>0.
\]
Since
\(
\{\phi(t_0,x_0)>u(t_0,x_0)\}
=\bigcup_{n\ge1}
\left\{\phi(t_0,x_0)-u(t_0,x_0)>\frac{1}{n}\right\},
\)
there exists $\delta>0$ such that the event
\[
B:=\left\{ \phi(t_0,x_0)-u(t_0,x_0)>\delta \right\}\in\sF_{t_0}
\]
has positive probability. We now construct a stochastic strict contact. 

\medskip

\textbf{Step 1.} Reduce the positive gap between $\phi$ and $u$ to the scalar process
\[
M_s :=\sup_{x\in\bR^+}\bigl(\phi(s,x)-u(s,x)\bigr)^+,\quad  s\ge t_0.
\]
By construction, $M_{t_0}\ge\delta$ on $B$. Condition \eqref{additional_condition_cp} gives $(\phi(s,\cdot)-u(s,\cdot))^+\in C_0(\bR^+)$, so the supremum is attained; continuity of $\phi-u$ in the function-space norm then implies continuity of $M$. By \eqref{condition:infty},  
\[
M_T
=\left(\sup_{x\in\bR^+}\bigl(\phi(T,x)-u(T,x)\bigr)\right)^+ \to 0
\]
almost surely as $T\to\infty$. Further, since 
\[
0\le M_T
\le \sup_{s,x}|\phi(s,x)|+\sup_{s,x}|u(s,x)| \in L^2(\Omega)
\]
dominated convergence yields convergence in $L^2(\Omega)$, and hence square integrability of the process $Y$ below. 

\medskip

\textbf{Step 2.} 
We construct a finite contact time $\tau_N$, an event $B_N\in\sF_{\tau_N}$ of positive probability, and an $\sF_{\tau_N}$-measurable maximizer $\xi_N\in\bR^+$ such that
\begin{equation}\label{max_attainability}
\phi(\tau_N,\xi_N)-u(\tau_N,\xi_N)=M_{\tau_N} > 0
\quad\text{on} \quad B_N.
\end{equation}
We introduce the continuous process on $[t_0,\infty]$ by 
\[
Y_s=M_s+\frac{\delta}{2}\bigl(e^{-rt_0}-e^{-rs}\bigr), \quad s\geq t_0,
\quad \text{and} \quad
Y_\infty=\lim_{T\to\infty}Y_T=\frac{\delta}{2}e^{-rt_0}.
\] 

Let $\mathcal T^s_\infty$ be the stopping times with values in $[s,\infty]$ and define the Snell envelope
\[
Z_s:=\operatorname{ess\,sup}_{\tau\in\mathcal T^s_\infty}
\bE_{\sF_s}[Y_\tau].
\]
By standard Snell-envelope results in
\cite[Appendix D, especially the infinite-horizon setup on p.~349,
Theorem D.7 on p.~354, and Theorem D.12 on p.~358]
{karatzas1998methods}, the stopping time
\[
\overline\tau:=\inf\{s\ge t_0:Y_s=Z_s\},
\qquad \inf\varnothing:=\infty,
\]
is optimal and
\begin{equation*}
Z_{t_0}=\bE_{\sF_{t_0}}[Y_{\overline\tau}]. 
\end{equation*}

Set $A:=\{\overline\tau<\infty\}$. We use two observations.
\begin{itemize}
	\item[(i)] On the set $A$ the Snell-envelope property yields
\[
Y_{\overline\tau}=Z_{\overline\tau}
\ge\bE_{\sF_{\overline\tau}}[Y_\infty]
=\frac{\delta}{2}e^{-rt_0}.
\]
Consequently,
\begin{equation*}
M_{\overline\tau}\ge\frac{\delta}{2}e^{-r\overline\tau}>0
\quad\text{on} \quad A.
\end{equation*}
\item[(ii)] The set $B\cap A$ has strictly positive probability. Indeed, on $B$, we have
\[
0< \delta-\frac{\delta}{2}e^{-rt_0} \leq Z_{t_0}-\bE_{\sF_{t_0}}[Y_\infty]
=\bE_{\sF_{t_0}}\left[
    (Y_{\overline\tau}-Y_\infty)\mathbf1_A
\right].
\]
Multiplying by $\mathbf1_B$ shows that $\bP(B\cap A)>0$; otherwise, the right-hand side would vanish on $B$. 
\end{itemize}

By (ii), 
there exists $N>t_0$ such that the set
\[
B_N:=B\cap\{\overline\tau\le N\} \subset B\cap A
\]
has positive probability. Let us set 
\[
	\tau_N:=\overline\tau\wedge N < \infty.
\]	
Then $B_N\in\sF_{\tau_N}$ and $\tau_N=\overline\tau$ on $B_N$; by (i),
\[
	M_{\tau_N} > 0 \quad \text{on} \quad B_N.
\]
Hence, on $B_N$, the continuous function $x\mapsto(\phi-u)(\tau_N,x)^+$ attains the positive maximum $M_{\tau_N}$. Measurable selection gives an $\sF_{\tau_N}$-measurable $\xi_N$ satisfying \eqref{max_attainability}.

\medskip

\textbf{Step 3.} The point $(\tau_N,\xi_N)$ gives pathwise contact, whereas Definition~\ref{def_vis_sol} requires conditional-expectation tangency over future stopping times. We obtain it by a martingale correction. For $s\ge0$, set
\begin{equation*}
\mathcal M_s:=\bE_{\sF_s}[Y_{\overline\tau}] \quad \text{and} \quad
\Phi(s,x)
:=\phi(s,x)+\frac{\delta}{2}\bigl(e^{-rt_0}-e^{-rs}\bigr)-\mathcal M_s.
\end{equation*}
Since $M_s$ is uniformly dominated by an $L^2(\Omega)$ random variable and $Y$ is of class D, we have $Y_{\overline\tau}\in L^2(\Omega)$. Hence, the martingale-representation theorem yields a 
representation integrand in $L^2(\Omega\times \bR^+)$ for $\mathcal M$. Moreover, $\mathcal M$ is independent of $x$, and $s\mapsto e^{-rs}$ has an $L^2$-integrable derivative. It follows that 
\[
	\Phi\in\mathcal C_{\sF}
	\quad \text{and} \quad 
(\Phi-u)(\tau_N,\xi_N)=0 ~~ \text{on} ~~ B_N.
\]
To construct the required stopping times we set
\[
\sigma_N:=\widehat\tau\wedge(\tau_N+1) \quad \text{where} \quad \widehat\tau:=\inf\{s\ge\overline\tau:M_s=0\},
\]
and for any $\widetilde\tau\in\mathcal T^{\tau_N}$, we put $\theta:=\widetilde\tau\wedge\sigma_N$. As
\(
\sup_{y\in\bR^+}(\phi-u)(\theta,y)\le M_\theta,
\)
the tower property and the Snell-envelope property yield, on $B_N$, that
\begin{align*}
\bE_{\sF_{\tau_N}}
\left[\sup_{y\in\bR^+}(\Phi-u)(\theta,y)\right]
\le \bE_{\sF_{\tau_N}}[Y_\theta]-\mathcal M_{\tau_N}
\le Z_{\tau_N}-Y_{\overline\tau}=0.
\end{align*}
Taking $\widetilde\tau=\tau_N$ gives equality. Thus
\(
\Phi\in\overline{\mathcal G}u(\tau_N,\xi_N;B_N).
\)
These stopping times are admissible in the definition of $\overline{\mathcal G}u$.
Applying the viscosity-supersolution property on $B_N$ and using \eqref{supersolution_condition_sup-clasical} gives
\begin{align*}
0
&\le\operatorname{ess}\limsup_{(s,x)\to(\tau_N^+,\xi_N)}
\left\{-\mathfrak d_s\Phi(s,x)
-\mathcal H(s,x,\partial_x\Phi(s,x))\right\}\\
&=-\frac{r\delta}{2}e^{-r\tau_N}
+\operatorname{ess}\limsup_{(s,x)\to(\tau_N^+,\xi_N)}
\left\{-\mathfrak d_s\phi(s,x)
-\mathcal H(s,x,\partial_x\phi(s,x))\right\}\\
&\le-\frac{r\delta}{2}e^{-r\tau_N}<0,
\end{align*}
leading to a contradiction. Hence $u\ge\phi$. 
\end{proof}

The comparison principle also identifies any classical solution with $\tilde V$.

\begin{defn}\label{defn-classical-solution}
A function $\phi\in \mathcal C_{\sF}\cap \mathcal S^2(C_0(\bR^+))$   is a classical solution to the stochastic Hamilton--Jacobi equation \eqref{SHJB_lim0} 
if it satisfies the following conditions:
\begin{enumerate}
    \item[(i)]  $\lim_{t\to\infty}\sup_{x\in\bR^+}|\phi(t,x)|=0$, a.s.;
    \item[(ii)] for $\bP\otimes dt$-a.e. $(\omega,t)\in\Omega\times[0,\infty)$,
        \begin{equation*}
      -\mathfrak{d}_t \phi(t,x)-\mathcal{H}(t, x, \partial_x \phi(t,x)) = 0,  \text{ for all } x\in\bR^+. 
    \end{equation*}
\end{enumerate}
\end{defn}

Theorem~\ref{comparison} immediately yields the following corollary.
\begin{cor}\label{corollary_classical_shjb}
    Let Assumption~\ref{assumption} hold. 
    If a classical solution to the stochastic Hamilton--Jacobi equation \eqref{SHJB_lim0} exists, then it is the unique viscosity solution and coincides with the value function $\tilde V$ defined in \eqref{dynamic-value-functional}. In particular, the classical solution is unique.
\end{cor}

We now turn to uniqueness. The proof uses a squeezing argument inspired by Perron's method. In stochastic HJ theory \cite{ekren2016viscosity-1,ekren2016viscosity-2,qiu_stochastic_hjb,qiu_uniqueness_2019}, regular barriers are typically built through piecewise Markovian approximations of the coefficients. We avoid this Markovianization by adding an independent Brownian perturbation $B$ to the state dynamics. The resulting upper and lower barriers are value functions of non-Markovian control problems with improved regularity.

\begin{thm}\label{uniqueness}
    Under Assumption~\ref{assumption}, the value function \eqref{dynamic-value-functional} is the unique viscosity solution to the stochastic Hamilton--Jacobi equation \eqref{SHJB_lim0}.
\end{thm} 
    
    \begin{proof}
We squeeze any viscosity solution between $\underline V^\epsilon$ and $\overline V^\epsilon$ using Theorem~\ref{comparison}, and then show that $\overline V^\epsilon-\underline V^\epsilon\to0$ as $\epsilon\to0$. Together with Theorem~\ref{vis_existence}, this proves uniqueness. In what follows $\check V$ denotes any viscosity solution of \eqref{SHJB_lim0}. By Definition~\ref{def_vis_sol}, 
\[
    \check V\in\mathcal S^2(C_0(\bR^+)) \quad \text{and} \quad \lim_{t\rightarrow \infty} \sup_{x\in\bR^+} |\check V(t,x)| =0.
\]

        \textbf{Step 1.} Let $B=\{B_t:t\ge0\}$ be a standard Brownian motion independent of the original $\sigma$-field $\sF$. By the usual product-space extension, we may work on a stochastic basis carrying both $W$ and $B$. Let $\widetilde{\sF}_t$ be the usual augmentation of
        \[
        \sF_t\vee\sigma(B_s:0\le s\le t),
        \]
        and let $\widetilde{\cA}$ denote the $[0,1]$-valued processes progressively measurable with respect to $(\widetilde{\sF}_t)_{t\ge0}$. Since $B$ is independent of the original $\sigma$-field,
        \[
        \bE\bigl[1_{\{\overline X>x\}}\mid\widetilde{\sF}_t\bigr]
        =\bE\bigl[1_{\{\overline X>x\}}\mid\sF_t\bigr]
        =\hat F(t,x).
        \]
        We identify the original random fields with their canonical lifts to the enlarged space. 
       On the enlarged space, ${\cal C}_{\widetilde F}$ is understood as in Definition \ref{cf1}, with $W$ replaced by the two-dimensional Brownian motion $(W,B)$  and
$\mathfrak{d}_{\omega}\phi$ correspondingly vector-valued. Conditioning with respect to the independent $B$-coordinate shows that the viscosity sub- and supersolution properties of $\check V$ are preserved under this independent enlargement; the proof of Theorem~\ref{comparison} is unchanged with $\sF$ replaced by $\widetilde{\sF}$. We therefore apply the comparison principle on the enlarged stochastic basis. For notational simplicity, expectations below are again denoted by $\bE$.

\medskip 

\textbf{Step 2.} 
     We next construct regular upper and lower barriers for the original value function. The construction has two ingredients, both introduced below. First, we will perturb the state dynamics by the independent Brownian motion \(\epsilon B\) and truncate the horizon at \(T^\epsilon=-\log(\epsilon)/r\). This leads to auxiliary value functions \(\underline U^\epsilon\) and \(\overline U^\epsilon\) associated with a uniformly superparabolic BSPDE and hence possessing the spatial regularity required for the comparison argument. Second, shifting these functions back to the original state variable will replace \(\hat F(t,x)\) by \(\hat F(t,x-\epsilon B_t)\) in the Hamiltonian. By the Lipschitz continuity of \(\hat F\), the resulting error is of order \(\epsilon L_{\hat F}|B_t|e^{-rt}\). We will compensate for this error by a linear BSDE with solution \(Y^\epsilon\).     By general BSDE theory (see, e.g., \cite{briand2003lp}), the linear backward stochastic differential equation
         \[
        Y^{\epsilon}_s =  \int_{s\land T^{\epsilon}}^{T^{\epsilon}} \! \epsilon L_{\hat F} |B_l| e^{-rl} dl - \int_{s\land T^{\epsilon}}^{T^{\epsilon}}\! Z^{\epsilon}_l dB_l, \quad s\in [0,\infty),
         \]
          admits a unique adapted solution $(Y^{\epsilon},Z^{\epsilon})$, equivalently characterized by
         \[
         -\mathfrak{d}_tY^{\epsilon}_t =  \epsilon L_{\hat F}|B_t| e^{-rt}, \quad t\in [0,T^{\epsilon}); \quad Y^{\epsilon}_{t}=0, \quad t\ge T^{\epsilon}.
         \]
         In particular,
         \[
         Y_t^\epsilon
         =\bE_{\widetilde{\sF}_t}\left[\int_{t\land T^\epsilon}^{T^\epsilon}
         \epsilon L_{\hat F}|B_s|e^{-rs}\,ds\right]\ge0.
         \]

To define the functions \(\underline U^\epsilon\) and \(\overline U^\epsilon\), we extend the survival function to $x<0$ by $\hat F(t,x)=\hat F(t,|x|)$. For $\epsilon>0$ and $(t,x)\in\bR^+\times\bR$, define
        \begin{align}
            \overline U^{\epsilon}(t,x) &= \esssup_{\alpha\in\widetilde{\cA}} \bE_{\widetilde{\sF}_{t}} \left[ \int_{t\land T^{\epsilon}}^{T^{\epsilon}}\!e^{-rs}\hat F(s,X_s^{t,x;\alpha,B})\alpha_s(1-\alpha_s)ds \right] + \frac{1}{4r}e^{-r(t\vee T^{\epsilon})},
            \label{upper-appr-U}\\
            \underline U^{\epsilon}(t,x) &= \esssup_{\alpha\in\widetilde{\cA}} \bE_{\widetilde{\sF}_{t}} \left[ \int_{t\land T^{\epsilon}}^{T^{\epsilon}}\!e^{-rs}\hat F(s,X_s^{t,x;\alpha,B})\alpha_s(1-\alpha_s)ds \right],
            \label{lower-appr-U}
        \end{align}
        where $X^{t,x;\alpha,B}$ satisfies the following stochastic differential equation
        \begin{equation*}
            X_s^{t,x;\alpha,B} = x + \int_t^s\!\alpha_r\,\ud r + \epsilon (B_s-B_t),\quad s\in[t,\infty).
        \end{equation*}
        Clearly, $\underline U^\epsilon\le\overline U^\epsilon$, and
        \[\lim_{x\to\infty}\underline U^\epsilon(t,x)=0,
\qquad
\lim_{x\to\infty}\overline U^\epsilon(t,x)
=\frac1{4r}e^{-r(t\vee T^\epsilon)}, \quad \text{for all } t\ge 0, \text{ a.s.}\]
The perturbation $\epsilon B$ makes the associated BSPDE superparabolic. By standard BSPDE theory, e.g.,~\cite{du2012p}, the equation
        \begin{equation*}
            \left\{
            \begin{aligned}
                -dV_{T^{\epsilon}}(t,x) &= \left[\frac{1}{2}\epsilon^2 \partial_{xx}^2V_{T^{\epsilon}}(t,x) + \mathcal H(t,x,\partial_x V_{T^{\epsilon}}(t,x))\right]dt - \psi(t,x)dW_t,
                \\
                V_{T^{\epsilon}}(T^{\epsilon},x) &= 0,
            \end{aligned}
            \right.
        \end{equation*}
        admits a unique adapted solution pair 
        $(V_{T^{\epsilon}},\psi)$ 
         satisfying, for each $p\in [2,\infty)$, 
        \[
        \bE\left[ \int_0^{T^{\epsilon}}\! \|V_{T^{\epsilon}}(t,\cdot)\|^p_{W^{2,p}} dt
        +  \int_0^{T^{\epsilon}}\! 
            \|\psi(t,\cdot)\|^2_{W^{1,2}} dt  + \sup_{t\in [0,T^{\epsilon}]} \|V_{T^{\epsilon}}(t,\cdot)\|^p_{W^{1,p}} \right] < \infty.
        \]
        By the Sobolev embedding theorem, we have $V_{T^{\epsilon}}\in L^p(\Omega\times [0,T^{\epsilon}];C^{1,1-\frac{1}{p}}(\bR))\cap L^p(\Omega;L^{\infty}(0,T^{\epsilon};C^{1-\frac{1}{p}}(\bR)))$, 
        and thus,
        \[
        \bE\left[ \int_0^{T^{\epsilon}}\! \|V_{T^{\epsilon}}(t,\cdot)\|^p_{C^{1,1-\frac{1}{p}}} dt + \sup_{t\in [0,T^{\epsilon}]} \|V_{T^{\epsilon}}(t,\cdot)\|^p_{C^{1-\frac{1}{p}}} \right] < \infty.
        \]
        The generalized It\^o--Kunita--Wentzell formula \cite[Theorem 3.1]{YangTang2013SdeRandomCoeff_BSPDE} and standard verification arguments as in \cite[Section 3.2]{peng_shjb} yield 
\[
 	V_{T^\epsilon} (t,x) = \underline U^{\epsilon}(t,x) \quad \mbox{for all $(t,x)\in[0,T^{\epsilon}]\times \bR$ a.s.}
\]
The barriers are therefore 
\[
\underline V^{\epsilon}(t,x)
=\underline U^{\epsilon}(t,x-\epsilon B_t)-Y^{\epsilon}_t,
\qquad
\overline V^{\epsilon}(t,x)
=\overline U^{\epsilon}(t,x-\epsilon B_t)+Y^{\epsilon}_t.
\] 
 
Fix \(p>2\) and choose \(1/p<\beta<1-1/p\). The \(W^{2,p}(\bR)\)-estimate and one-dimensional Sobolev embedding give the spatial \(C^\beta\)-regularity in Definition~\ref{cf1}(ii), while the corresponding \(L^p\)-bounds yield the required \(L^2\)-bounds on \([0,T^\epsilon]\). The extensions beyond \(T^\epsilon\) and the BSDE estimates for \(Y^\epsilon\) give global time-integrability. Hence 
 \[
  	\underline V^\epsilon,\overline V^\epsilon\in\mathcal C_{\widetilde{\sF}}.
\]

Applying the It\^o--Kunita--Wentzell--Krylov formula \cite{krylov2011ito} to $\underline U^{\epsilon}(t,x-\epsilon B_t)$ gives
        \begin{equation*}
            \left\{
            \begin{aligned}
                -\mathfrak{d}_t \underline V^{\epsilon}(t,x) &= \mathcal H(t,x - \epsilon B_t,\partial_x \underline U^{\epsilon}(t,x - \epsilon B_t)) -\epsilon L_{\hat F} |B_t| e^{-rt}
                \\
                &=  \mathcal H(t,x - \epsilon B_t,\partial_x \underline V^{\epsilon}(t,x)) -\epsilon L_{\hat F} |B_t| e^{-rt}
                \\
                \underline V^{\epsilon}(T^{\epsilon},x) &= 0.
            \end{aligned}
            \right.
        \end{equation*}
        For $t\ge T^{\epsilon}$, we have $\underline V^{\epsilon}(t,x)=0$, and thus simple calculations yield 
        \begin{align*}
              -\mathfrak{d}_t \underline V^{\epsilon}(t,x) - \mathcal H(t,x,\partial_x \underline V^{\epsilon}(t,x)) 
             &= -\frac{1}{4}e^{-rt}\hat F(t,x)\le 0, \quad \text{for all }x\in\bR^+,\ \bP\otimes dt\text{-a.e. on }\Omega\times[T^{\epsilon},\infty).
        \end{align*}
        When $t<T^{\epsilon}$, we have 
        \begin{align*}
             & -\mathfrak{d}_t \underline V^{\epsilon}(t,x) - \mathcal H(t,x,\partial_x \underline V^{\epsilon}(t,x)) 
            \\
            &=
             \underbrace{-\mathfrak{d}_t \underline V^{\epsilon}(t,x) - \mathcal H(t,x - \epsilon B_t,\partial_x \underline V^{\epsilon}(t,x)) + \epsilon L_{\hat F} |B_t| e^{-rt}}_{=0}
             \\
             &\quad 
             +\underbrace{\mathcal H(t,x - \epsilon B_t,\partial_x \underline V^{\epsilon}(t,x)) 
             - \mathcal H(t,x,\partial_x \underline V^{\epsilon}(t,x)) 
              - \epsilon L_{\hat F} |B_t| e^{-rt} }_{\leq 0}
             \\
             &\le 0, \quad \text{for all }x\in\bR^+,\ \bP\otimes dt\text{-a.e. on }\Omega\times[0,T^{\epsilon}),
        \end{align*}
        where we used the fact that
        \[
        |\mathcal H(t,x,q)-\mathcal H(t,y,q)|
\le \frac14e^{-rt}L_{\hat F}|x-y|
\le e^{-rt}L_{\hat F}|x-y|, \quad \text{for all } (t,x,y,q)\in [0,\infty)\times\bR\times\bR\times\bR, \text{ a.s.}
\]
        Moreover, since $\check V(t,\cdot)\in C_0(\bR^+)$ and $Y_t^\epsilon\ge0$,
        \[
        \limsup_{x\to\infty}\bigl(\underline V^\epsilon(t,x)-\check V(t,x)\bigr)
        =-Y_t^\epsilon\le0.
        \]
        For $T\ge T^\epsilon$, $\underline V^\epsilon(T,x)=0$, and the uniform decay of $\check V$ yields
        \[
        \limsup_{T\to\infty}\sup_{x\in\bR^+}
        \bigl(\underline V^\epsilon(T,x)-\check V(T,x)\bigr)\le0.
        \]
        Hence Theorem~\ref{comparison} gives 
         \[
         	\underline V^{\epsilon}(t,x)\leq \check V(t,x) \qquad\text{for all } (t,x)\in [0,\infty)\times\bR^+ \quad\text{a.s.}
	\]
        Similarly, the upper barrier satisfies the required temporal asymptotic condition:
        \[
        \liminf_{T\to\infty}\inf_{x\in\bR^+}
        \bigl(\overline V^\epsilon(T,x)-\check V(T,x)\bigr)\ge0.
        \]
        For $t\ge T^\epsilon$, $\overline V^\epsilon(t,x)=\frac{1}{4r}e^{-rt}$, so
        \begin{equation*}
            -\mathfrak{d}_t \overline V^{\epsilon}(t,x) 
            - \mathcal H(t,x,\partial_x\overline V^{\epsilon}(t,x)) 
            =\frac{1}{4}e^{-rt} - \frac{1}{4}e^{-rt}\hat F(t,x) \ge 0, \quad \text{for all }x\in\bR^+,\ \bP\otimes dt\text{-a.e. on }\Omega\times[T^{\epsilon},\infty).
        \end{equation*}
        For $t<T^\epsilon$,
        \begin{align*}
            &-\mathfrak{d}_t \overline V^{\epsilon}(t,x)
            - \mathcal H(t,x,\partial_x\overline V^{\epsilon}(t,x))
            \\
            &= \underbrace{-\mathfrak{d}_t \overline V^{\epsilon}(t,x) - \mathcal H(t,x-\epsilon B_t,\partial_x\overline V^{\epsilon}(t,x)) -\epsilon L_{\hat F} |B_t| e^{-rt}}_{=0}
            \\ 
            &\quad 
            +
             \underbrace{
                 \mathcal H(t,x-\epsilon B_t,\partial_x\overline V^{\epsilon}(t,x))
             -\mathcal H(t,x,\partial_x\overline V^{\epsilon}(t,x)) 
             + \epsilon L_{\hat F} |B_t|e^{-rt}}_{\geq 0} 
            \\
            &\geq 0, \quad \text{for all }x\in\bR^+,\ \bP\otimes dt\text{-a.e. on }\Omega\times[0,T^{\epsilon}).
        \end{align*}
        Since $\check V(t,\cdot)\in C_0(\bR^+)$ and
        \[
        \lim_{x\to\infty}\overline V^\epsilon(t,x)
        =\frac{1}{4r}e^{-r(t\vee T^\epsilon)}+Y_t^\epsilon\ge0,
        \]
        we also have
        \[
        \liminf_{x\to\infty}\bigl(\overline V^\epsilon(t,x)-\check V(t,x)\bigr)\ge0.
        \]
        Theorem~\ref{comparison} therefore yields 
        \[
        	\overline V^{\epsilon}(t,x)\geq \check V(t,x) \qquad\text{for all } (t,x)\in [0,\infty)\times\bR^+ \quad\text{a.s.}
        \]
        
        \medskip
        
        \textbf{Step 3.} We estimate the gap between the two barriers. 
        By definition, for each $(t,x)\in[0,\infty)\times\bR^+$,
        \begin{align*}
             \bE|\overline V^{\epsilon}(t,x) - \underline V^{\epsilon} (t,x) | 
            \leq  \bE\left[ 2 |Y_t^{\epsilon}|  + \frac{1}{4r}e^{-rT^{\epsilon}}\right]
            &\leq \bE\left[ 2 \epsilon L_{\hat F} \int_0^{\infty}e^{-rs}|B_s|ds +  \frac{\epsilon}{4r} \right]  \\
            &\leq \left(\frac{2L_{\hat F}}{r^{3/2}} + \frac{1}{4r}\right)\epsilon  
            \\
            &\rightarrow 0, \quad \text{as } \epsilon\to 0^+.
        \end{align*}
        Both $\tilde V$ and $\check V$ lie between the same barriers, hence for each fixed $(t,x)$,
\[
\bE|\check V(t,x)-\tilde V(t,x)|
\le
\bE[\overline V^\epsilon(t,x)-\underline V^\epsilon(t,x)]
\longrightarrow 0.
\]
Thus $\tilde V=\check V$ a.s. on rational $(t,x)$, and pathwise continuity extends the equality to $\bR^+\times\bR^+$.
\end{proof}

\subsection{Markovian reference model}

The random survival function $\hat F$ generally leads to a non-Markovian stochastic HJ equation. Under the Markovian specification
\[ 
\hat F(\omega,t,x)=\hat f\bigl(t,x,Y_t(\omega)\bigr),
\]
for some sufficiently regular function $\hat f:[0,\infty)\times\bR\times\bR\to\bR$, where $Y$ is a diffusion satisfying
\begin{equation*} 
Y_s = y + \int_0^s\!\beta(\tau,Y_\tau)\,\ud\tau + \int_0^s\!\gamma(\tau,Y_\tau)\,\ud W_\tau,
\quad s\ge0.
\end{equation*}
To emphasize this distinction through the explicit $\omega$-dependence, the Markovian property can be recovered by taking the pair $(X,Y)$ as the state process in the control problem. Define the deterministic Hamiltonian
\[
\mathcal H^{\rm M}(t,x,y,p)
:=\max_{a\in[0,1]}
\left\{ap+e^{-rt}a(1-a)\hat f(t,x,y)\right\},
\]
where, for the approximating problems on $\bR$, $\hat f$ is understood with the same even extension in the $x$-variable as $\hat F$ above. In \eqref{upper-appr-U} and \eqref{lower-appr-U}, standard Markovian control theory indicates that
\[
\overline U^{\epsilon}(\omega,t,x)=\overline u^{\epsilon}(t,x,Y_t(\omega)),
\qquad
\underline U^{\epsilon}(\omega,t,x)=\underline u^{\epsilon}(t,x,Y_t(\omega)),
\]
and
\[
\mathfrak d_\omega\overline U^{\epsilon}(\omega,t,x)
=\gamma(t,Y_t(\omega))\partial_y\overline u^{\epsilon}(t,x,Y_t(\omega)),
\qquad
\mathfrak d_\omega\underline U^{\epsilon}(\omega,t,x)
=\gamma(t,Y_t(\omega))\partial_y\underline u^{\epsilon}(t,x,Y_t(\omega)).
\]
The deterministic functions $\overline u^\epsilon$ and $\underline u^\epsilon$ satisfy
\begin{align*}
-\partial_t \overline u^{\epsilon}
-\frac12\gamma^2(t,y)\partial_{yy}^2\overline u^{\epsilon}
-\beta(t,y)\partial_y\overline u^{\epsilon}
-\frac{\epsilon^2}{2}\partial_{xx}^2\overline u^{\epsilon}
-\mathcal H^{\rm M}(t,x,y,\partial_x\overline u^{\epsilon})&=0,\\
-\partial_t \underline u^{\epsilon}
-\frac12\gamma^2(t,y)\partial_{yy}^2\underline u^{\epsilon}
-\beta(t,y)\partial_y\underline u^{\epsilon}
-\frac{\epsilon^2}{2}\partial_{xx}^2\underline u^{\epsilon}
-\mathcal H^{\rm M}(t,x,y,\partial_x\underline u^{\epsilon})&=0,
\end{align*}
with terminal conditions
\[
\overline u^{\epsilon}(T^{\epsilon},x,y)=\frac{1}{4r}e^{-rT^{\epsilon}},
\qquad
\underline u^{\epsilon}(T^{\epsilon},x,y)=0.
\]

The original control problem is Markovian in $(X,Y)$ as well. Hence its value function has the representation
\[
\tilde V(\omega,t,x)=\tilde v(t,x,Y_t(\omega))
\]
for a deterministic function $\tilde v$. Moreover,
\[
\mathfrak d_\omega\tilde V(\omega,t,x)
=\gamma(t,Y_t(\omega))\partial_y\tilde v(t,x,Y_t(\omega)).
\]
The stochastic HJ equation \eqref{SHJB_lim0} therefore becomes
\begin{align*}
-\mathfrak d_t\tilde V(\omega,t,x)
-\mathcal H^{\rm M}\bigl(t,x,Y_t(\omega),\partial_x\tilde v(t,x,Y_t(\omega))\bigr)=0,
\end{align*}
or, equivalently,
\begin{align*}
-d\tilde v(t,x,Y_t)
=\mathcal H^{\rm M}(t,x,Y_t,\partial_x\tilde v(t,x,Y_t))\,dt
-\gamma(t,Y_t)\partial_y\tilde v(t,x,Y_t)\,dW_t.
\end{align*}
By It\^o's formula, $\tilde v$ solves the deterministic PDE
\begin{align*}
-\partial_t\tilde v(t,x,y)
-\frac12\gamma^2(t,y)\partial_{yy}^2\tilde v(t,x,y)
-\beta(t,y)\partial_y\tilde v(t,x,y)
-\mathcal H^{\rm M}(t,x,y,\partial_x\tilde v(t,x,y))=0,
\end{align*}
for $(t,x,y)\in\bR^+\times\bR^+\times\bR$. In particular,
\begin{align*}
\mathfrak d_t\tilde V(\omega,t,x)
&=\partial_t\tilde v(t,x,Y_t(\omega))
+\beta(t,Y_t(\omega))\partial_y\tilde v(t,x,Y_t(\omega))\\
&\quad+\frac12\gamma^2(t,Y_t(\omega))\partial_{yy}^2\tilde v(t,x,Y_t(\omega)).
\end{align*}

\section{Asymptotic analysis with deterministic controls}\label{subsection_special_case}


We now analyze the long-run behavior of the value function and optimal extraction policy through deterministic benchmark problems. Although the underlying stochastic HJ equation with random, path-dependent coefficients is not explicitly solvable, its asymptotic behavior admits a tractable pathwise characterization. Economically, this identifies residual uncertainty about large reserve levels—and, in particular, the tail hazard rate of the limiting reserve distribution—as a key determinant of long-run extraction.

\subsection{Deterministic control problem}
\label{subsec:well-posedness}

We first consider the benchmark case in which the survival function is deterministic and time-independent. This model will also be used in the asymptotic analysis below.

\begin{thm}\label{thm_deterministic}
Under Assumption~\ref{assumption}, suppose that
\[
\hat F(\omega,t,x)\equiv \tilde F(x)
\]
for some deterministic function $\tilde F$ on $\bR^+$. Then:
\begin{enumerate}
\item[(i)] The unique viscosity solution $\tilde V$ of \eqref{SHJB_lim0} is classical and has the form
\[
\tilde V(t,x)=e^{-rt}\hat V(x),
\]
where $\hat V$ is the classical solution of
\begin{equation}\label{HJ-eq-equiv}
r\hat V(x)=\frac{\left|\left(\partial_x\hat V(x)+\tilde F(x)\right)^+\right|^2}{4\tilde F(x)},
\qquad x\in\bR^+,
\qquad \lim_{x\to\infty}\hat V(x)=0.
\end{equation}

\item[(ii)] The function $\hat V$ is the value function of the deterministic control problem
\begin{equation}\label{control_prob_deterministic}
\hat V(x)=\max_{\alpha\in\cA_d}\int_0^\infty e^{-rs}\alpha_s(1-\alpha_s)
\tilde F(X_s^{0,x;\alpha})\,\ud s,
\qquad x\in\bR^+,
\end{equation}
where $\cA_d$ denotes the deterministic $[0,1]$-valued controls. The optimal feedback is
\begin{equation}\label{optimal_feedback_deterministic}
a^*(x)=\frac{\left(\partial_x\hat V(x)+\tilde F(x)\right)^+}{2\tilde F(x)}
=\sqrt{\frac{r\hat V(x)}{\tilde F(x)}},
\qquad x\in\bR^+.
\end{equation}
For initial state $x$, the corresponding optimal control satisfies
\[
\alpha_s^{*,x}=a^*(X_s^{0,x;\alpha^{*,x}}).
\]
\end{enumerate}
\end{thm}

\begin{proof}
\textbf{Step 1.} We first solve the boundary-value problem
\begin{equation}\label{HJ-eq-backward}
\psi'(x)=\sqrt{4r\psi(x)\tilde F(x)}-\tilde F(x),
\qquad x\ge0,
\qquad \lim_{x\to\infty}\psi(x)=0.
\end{equation}
For $T>0$, consider
\[
\psi_T'(x)=2\sqrt{r\tilde F(x)(\psi_T(x))^+}-\tilde F(x),
\qquad 0\le x\le T,
\qquad \psi_T(T)=0.
\]
By Peano's theorem, a local solution exists near $T$. To obtain the required bounds, set $u_T(s):=\psi_T(T-s)$ and $G_T(s):=\tilde F(T-s)$. Then $G_T$ is non-decreasing and
\[
u_T'(s)=G_T(s)-2\sqrt{rG_T(s)(u_T(s))^+},
\qquad u_T(0)=0.
\]
At $u_T=0$ the drift is strictly positive, while at $u_T=G_T/(4r)$ it vanishes and the upper barrier $G_T/(4r)$ is non-decreasing. Hence
\[
0\le u_T(s)\le\frac{G_T(s)}{4r},
\]
or equivalently
\[
0\le \psi_T(x)\le \frac{\tilde F(x)}{4r},
\qquad
-\tilde F(x)\le\psi_T'(x)\le0.
\]
These bounds prevent finite-time blow-up, so $\psi_T$ extends to $[0,T]$. The family $\{\psi_T\}_{T>0}$ is therefore uniformly bounded and equicontinuous on compact intervals. By Arzel\`a--Ascoli and a diagonal argument, along some $T_n\uparrow\infty$,
\[
\psi_{T_n}\to\psi
\]
locally uniformly, where
\[
\psi(x)=\psi(0)+\int_0^x\left(2\sqrt{r\tilde F(s)\psi(s)}-\tilde F(s)\right)\ud s.
\]
Thus $\psi\in C^1([0,\infty))$ solves \eqref{HJ-eq-backward} and satisfies
\[
0\le\psi(x)\le\frac{\tilde F(x)}{4r}.
\]
Since $\tilde F(x)\to0$, we have $\psi(x)\to0$. Moreover, $\tilde F>0$ implies $\psi>0$ on finite intervals. As $\psi$ and $\tilde F$ are locally bounded away from zero and locally Lipschitz, the ODE implies that $\psi'$ is locally Lipschitz.

\medskip
\textbf{Step 2.} Passing the derivative bounds to the limit gives
$-\tilde F\le\psi'\le0$, so $\psi$ solves \eqref{HJ-eq-equiv}. Set
\[
\tilde V(t,x)=e^{-rt}\psi(x).
\]
For every $\alpha\in[0,1]$, completing the square yields
\begin{align*}
&\alpha\partial_x\tilde V(t,x)
+\alpha(1-\alpha)e^{-rt}\tilde F(x)\\
&\qquad=
-e^{-rt}\tilde F(x)
\left|\alpha-\frac{\left(\psi'(x)+\tilde F(x)\right)^+}{2\tilde F(x)}\right|^2
+e^{-rt}\frac{\left|\left(\psi'(x)+\tilde F(x)\right)^+\right|^2}{4\tilde F(x)}.
\end{align*}
Taking the supremum over $\alpha\in[0,1]$ and using \eqref{HJ-eq-equiv},
\[
-\mathfrak d_t\tilde V(t,x)
=-\partial_t\tilde V(t,x)
=\mathcal H(t,x,\partial_x\tilde V(t,x)).
\]
The preceding bounds and local regularity show that $\tilde V$ belongs to the class of Definition~\ref{defn-classical-solution}. Corollary~\ref{corollary_classical_shjb} therefore identifies it with the unique viscosity solution, proving (i).

For (ii), define
\[
a^*(x)=\sqrt{r\psi(x)/\tilde F(x)}.
\]
By \eqref{HJ-eq-backward}, this agrees with \eqref{optimal_feedback_deterministic}, and $0<a^*(x)\le1/2$. Since $a^*$ is locally Lipschitz, the closed-loop equation
\[
\dot X_s=a^*(X_s),\qquad X_0=x,
\]
has a unique global solution. Hence the feedback control belongs to $\cA_d$. For any $\alpha\in\cA_d$,
\begin{align*}
\int_0^\infty e^{-rs}\alpha_s(1-\alpha_s)\tilde F(X_s^{0,x;\alpha})\,\ud s
&\le \int_0^\infty e^{-rs}
\left[r\psi(X_s^{0,x;\alpha})-\alpha_s\psi'(X_s^{0,x;\alpha})\right]\ud s\\
&=\psi(x)-\lim_{t\to\infty}e^{-rt}\psi(X_t^{0,x;\alpha})
=\psi(x).
\end{align*}
Equality holds for the feedback control $\alpha_s=a^*(X_s^{0,x;\alpha})$ a.e., which proves (ii).
\end{proof}

\begin{cor}\label{corollary_deterministic_classical}
Under the assumptions of Theorem~\ref{thm_deterministic}, the HJ equation \eqref{HJ-eq-equiv} and the ODE \eqref{HJ-eq-backward} have the same unique classical solution $\hat V$, satisfying
\[
0<\hat V(x)<\frac{\tilde F(x)}{4r},
\qquad
-\tilde F(x)<\partial_x\hat V(x)<0,
\qquad x\in\bR^+.
\]
\end{cor}

\subsection{Asymptotic behavior of the value function and the optimal strategy}

We now study the case in which uncertainty about the resource level persists in the long run. For each fixed $x\in\bR^+$, the martingale convergence theorem gives
\[
\hat F(t,x)=\bE\!\left[1_{\{\overline X>x\}}\mid\sF_t\right]
\longrightarrow
\bE\!\left[1_{\{\overline X>x\}}\mid\sF_\infty\right]
=:\overline F(x).
\]
To obtain a sufficiently regular limiting problem, we impose the following stronger uniform convergence assumption.

\begin{ass}\label{assumption_asym}
There exists a jointly measurable function
\[
\overline F:(\Omega\times\bR^+,\sF_\infty\otimes\cB(\bR^+))
\to((0,\infty),\cB((0,\infty)))
\]
and a constant $L_F \geq 0$ such that
\begin{align*}
\lim_{t\to\infty}\bE\left[\sup_{x\in\bR^+}|\hat F(t,x)-\overline F(x)|\right]&=0,\\
|\overline F(x)-\overline F(y)|&\le L_{\overline F}|x-y|,
\qquad x,y\in\bR^+.
\end{align*}
\end{ass}

\begin{rmk}
Since $0\le\hat F\le1$ and $\overline F>0$ by assumption, the first condition implies, after passing to an a.s.~uniformly convergent subsequence, that $0<\overline F\le1$ and that $\overline F(\omega,\cdot)$ is non-increasing and vanishes at infinity for a.e. $\omega$.
\end{rmk}

The Lipschitz requirement excludes the fully revealing limit $\overline F(x)=1_{\{\overline X>x\}}$ and therefore describes a setting with persistent long-run uncertainty.

\subsubsection{Asymptotics of the value function}

The following proposition identifies the long-run limit of the rescaled value function.

\begin{prop}\label{prop_asym}
Under Assumptions \ref{assumption} and \ref{assumption_asym},
\begin{equation}\label{value-asymptotic-limit}
e^{rt}\tilde V(t,x)\longrightarrow\hat V(x)
\quad\text{in }L^1(\Omega),
\qquad x\in\bR^+,
\end{equation}
where, for almost every $\omega$, $\hat V(\omega,\cdot)$ is the unique classical solution of
\begin{equation}\label{HJ-eq-asymptotic}
r\hat V(x)=
\frac{\left|\left(\partial_x\hat V(x)+\overline F(x)\right)^+\right|^2}{4\overline F(x)},
\qquad x\in\bR^+,
\qquad \lim_{x\to\infty}\hat V(x)=0.
\end{equation}
\end{prop}

\begin{proof}
\textbf{Step 1.} We first identify $\hat V$ as the pathwise value function associated with $\overline F$ and verify measurability. For any bounded nonnegative function $f$, initial state $x$, and deterministic control $a:[0,\infty)\to[0,1]$, define
\[
\Gamma(f,x;a)
:=\int_0^\infty e^{-ru}a_u(1-a_u)
 f\left(x+\int_0^u a_v\,\ud v\right)\ud u,
\qquad
v(f,x):=\sup_a\Gamma(f,x;a).
\]
For bounded nonnegative $f,g$,
\begin{equation}\label{value-coefficient-stability}
|\Gamma(f,x;a)-\Gamma(g,x;a)|\le\frac{\|f-g\|_\infty}{4r},
\qquad
|v(f,x)-v(g,x)|\le\frac{\|f-g\|_\infty}{4r}.
\end{equation}
Applying Theorem~\ref{thm_deterministic} pathwise with $\tilde F=\overline F(\omega,\cdot)$ gives
\[
\hat V(\omega,x):=v(\overline F(\omega,\cdot),x),
\]
which is the unique classical solution of \eqref{HJ-eq-asymptotic}; moreover $0\le\hat V\le1/(4r)$. To verify $\sF_\infty$-measurability, let
\[
\mathbb U:=\left\{a\in L^1(\bR^+,e^{-ru}\,\ud u):0\le a\le1\ \text{a.e.}\right\},
\qquad
\|a\|_r:=\int_0^\infty e^{-ru}|a_u|\,\ud u.
\]
This is a Polish space. If $f$ is $L$-Lipschitz and $0\le f\le1$, then for $a,b\in\mathbb U$,
\begin{align*}
|\Gamma(f,x;a)-\Gamma(f,x;b)|
&\le \int_0^\infty e^{-ru}|a_u-b_u|\,\ud u
+\frac L4\int_0^\infty e^{-ru}\int_0^u|a_v-b_v|\,\ud v\,\ud u\notag\\
&=\left(1+\frac{L}{4r}\right)\|a-b\|_r.
\end{align*}
Hence $a\mapsto\Gamma(f,x;a)$ is continuous. Choosing a countable dense family $(a^n)_{n\ge1}\subset\mathbb U$ of rational-valued step functions with rational breakpoints and compact support,
\begin{equation}\label{countable-control-reduction}
v(f,x)=\sup_{n\ge1}\Gamma(f,x;a^n).
\end{equation}
Joint measurability of $\overline F$ therefore implies that $\hat V(\cdot,x)$ is $\sF_\infty$-measurable for every $x$.

\medskip
\textbf{Step 2.} We compare the control problems associated with $\hat F$ and $\overline F$. After the change of variables $s=t+u$,
\begin{equation*}
e^{rt}\tilde V(t,x)
=\esssup_{\alpha\in\cA}
\bE_{\sF_t}\!\left[
\int_0^\infty e^{-ru}\alpha_{t+u}(1-\alpha_{t+u})
\hat F(t+u,X_{t+u}^{t,x;\alpha})\,\ud u\right].
\end{equation*}
Set
\[
\Phi_t(x):=\esssup_{\alpha\in\cA}
\bE_{\sF_t}\!\left[
\int_0^\infty e^{-ru}\alpha_{t+u}(1-\alpha_{t+u})
\overline F(X_{t+u}^{t,x;\alpha})\,\ud u\right],
\]
and
\[
\delta_t:=\sup_{y\in\bR^+}|\hat F(t,y)-\overline F(y)|\le1.
\]
Then
\begin{equation}\label{first-comparison}
\bE\big[|e^{rt}\tilde V(t,x)-\Phi_t(x)|\big]
\le\frac14\int_0^\infty e^{-ru}\bE[\delta_{t+u}]\,\ud u
\longrightarrow0
\end{equation}
by dominated convergence. It therefore remains to show that $\Phi_t(x)\to\hat V(x)$ in $L^1$. For every admissible control, its realized time path is admissible in the definition of $v(\overline F,x)$, hence
\begin{equation}\label{full-information-upper-bound}
\Phi_t(x)\le\bE_{\sF_t}[\hat V(x)].
\end{equation}
For the converse, fix $\varepsilon>0$. With the dense family from \eqref{countable-control-reduction}, define
\[
Z_n^t:=\Gamma(\hat F(t,\cdot),x;a^n),
\qquad
Z^t:=\sup_{n\ge1}Z_n^t=v(\hat F(t,\cdot),x),
\]
and
\[
N^{t,\varepsilon}:=
\min\{n\ge1:Z_n^t\ge Z^t-\varepsilon\}.
\]
Since each $Z_n^t$ is $\sF_t$-measurable, so is $N^{t,\varepsilon}$. Thus
\[
a^{t,\varepsilon}(\omega,u):=a_u^{N^{t,\varepsilon}(\omega)}
\]
defines an admissible post-$t$ control and satisfies
\[
\Gamma(\hat F(t,\cdot),x;a^{t,\varepsilon})
\ge v(\hat F(t,\cdot),x)-\varepsilon.
\]
Applying \eqref{value-coefficient-stability} twice,
\[
\Gamma(\overline F,x;a^{t,\varepsilon})
\ge \hat V(x)-\varepsilon-\frac{\delta_t}{2r},
\]
and therefore
\[
\Phi_t(x)\ge
\bE_{\sF_t}\!\left[\hat V(x)-\varepsilon-\frac{\delta_t}{2r}\right].
\]
Together with \eqref{full-information-upper-bound}, this gives
\begin{equation*}
\bE\!\left[\bE_{\sF_t}[\hat V(x)]-\Phi_t(x)\right]
\le\frac{\bE[\delta_t]}{2r}\longrightarrow0.
\end{equation*}
Since $\hat V(x)$ is bounded and $\sF_\infty$-measurable, martingale convergence yields
$\bE_{\sF_t}[\hat V(x)]\to\hat V(x)$ in $L^1$. Hence $\Phi_t(x)\to\hat V(x)$ in $L^1$, and \eqref{first-comparison} proves \eqref{value-asymptotic-limit}.
\end{proof}

\begin{rmk}
If information updating ceases at some finite $t_0$, so that
$\hat F(t,x)=\hat F(t_0,x)$ for all $t\ge t_0$, then Proposition~\ref{prop_asym} applies with
$\overline F(x)=\hat F(t_0,x)$. Consequently,
\[
\tilde V(t,x)=e^{-r(t-t_0)}\tilde V(t_0,x)=e^{-rt}\hat V(x),
\qquad t\ge t_0,
\]
where $\hat V$ is the unique pathwise classical solution of \eqref{HJ-eq-asymptotic}.
\end{rmk}

\subsubsection{Asymptotics of the optimal control}

We next study the large-$x$ behavior of the optimal feedback associated with the limiting value function $\hat V$. This requires additional regularity of the limiting survival function.

\begin{ass}\label{F_c1}
The limiting survival function satisfies
$\overline F(\omega,\cdot)\in C^1(\bR^+)$ for almost every $\omega$.
\end{ass}

Let us now define the pathwise hazard rate
\[
h(\omega,x):=-\frac{\partial_x\overline F(\omega,x)}{\overline F(\omega,x)},
\qquad x\in\bR^+.
\]
Under Assumptions \ref{assumption_asym} and \ref{F_c1}, $\overline F(\omega,\cdot)$ is positive and non-increasing, so $h\ge0$ a.e. The next proposition identifies the limiting feedback through the asymptotic hazard rate.

\begin{prop}\label{prop_asym_con}
Suppose that Assumptions \ref{assumption}, \ref{assumption_asym}, and \ref{F_c1} hold and, for almost every $\omega$,
\[
\lim_{x\to\infty}h(\omega,x)=\lambda(\omega)\in[0,\infty].
\]
Let $\alpha^*$ be the optimal feedback associated with $\hat V$. Then, for a.e. $\omega$,
\[
\lim_{x\to\infty}\alpha^*(\omega,x)=
\begin{cases}
\displaystyle\frac{1}{1+\sqrt{1+\lambda(\omega)/r}},&\lambda(\omega)<\infty,\\[2mm]
0,&\lambda(\omega)=\infty.
\end{cases}
\]
\end{prop}

\begin{proof}
Fix $\omega$ outside a null set and suppress it from the notation. For $\alpha\in\cA_d$, set
\[
Y_s^\alpha:=\int_0^s\alpha_u\,\ud u,
\qquad
X_s^{0,x;\alpha}=x+Y_s^\alpha,
\]
and define the normalized value
\begin{equation*}
U(x):=\frac{\hat V(x)}{\overline F(x)}
=\sup_{\alpha\in\cA_d}\int_0^\infty e^{-rs}\alpha_s(1-\alpha_s)
\frac{\overline F(x+Y_s^\alpha)}{\overline F(x)}\,\ud s.
\end{equation*}
By \eqref{HJ-eq-asymptotic},
\begin{equation}\label{normalized-feedback-formula}
\alpha^*(x)
=\frac{\big(\partial_x\hat V(x)+\overline F(x)\big)^+}{2\overline F(x)}
=\sqrt{\frac{r\hat V(x)}{\overline F(x)}}
=\sqrt{rU(x)}.
\end{equation}

For $q\ge0$, consider the exponential benchmark $f_q(x)=e^{-qx}$ and define
\[
U_q:=\sup_{\alpha\in\cA_d}\int_0^\infty e^{-rs}\alpha_s(1-\alpha_s)e^{-qY_s^\alpha}\,\ud s,
\qquad
c(q):=\frac{1}{1+\sqrt{1+q/r}}.
\]
Its stationary HJ equation is
\[
rU_q=\max_{a\in[0,1]}\{a(1-a)-qaU_q\},
\]
where the case $q=0$ gives $U_0=1/(4r)$ and $\alpha_0^*=1/2$. Direct substitution gives
\begin{equation}\label{exponential-benchmark-value}
U_q=\frac{c(q)^2}{r},
\qquad
\alpha_q^*=c(q).
\end{equation}

Suppose that for some $0\le a\le b<\infty$ there exists $R>0$ such that
\[
a\le h(z)\le b,
\qquad z\ge R.
\]
Then for $x\ge R$ and $y\ge0$,
\begin{equation*}
e^{-by}\le\frac{\overline F(x+y)}{\overline F(x)}\le e^{-ay}.
\end{equation*}
Substituting $y=Y_s^\alpha$, integrating, and taking suprema yields
\[
U_b\le U(x)\le U_a,
\qquad x\ge R.
\]
Hence, by \eqref{normalized-feedback-formula} and \eqref{exponential-benchmark-value},
\begin{equation*}
c(b)\le\alpha^*(x)\le c(a),
\qquad x\ge R.
\end{equation*}

If $\lambda<\infty$, then for every $\varepsilon>0$ and all sufficiently large $x$,
\[
(\lambda-\varepsilon)^+\le h(x)\le\lambda+\varepsilon.
\]
Thus
\[
c(\lambda+\varepsilon)
\le\alpha^*(x)
\le c((\lambda-\varepsilon)^+),
\]
and continuity of $c$ gives $\alpha^*(x)\to c(\lambda)$. If $\lambda=\infty$, then for every $M>0$, eventually $h(x)\ge M$, so
\[
0\le\limsup_{x\to\infty}\alpha^*(x)\le c(M)\to0
\qquad\text{as }M\to\infty.
\]
\end{proof}

\subsubsection{Numerical illustrations}

We finally illustrate the asymptotic behavior of $\alpha^*$ for four deterministic survival functions. We set $r=0.1$ and consider the following pairs $(\overline F_i,\hat V_i)$ solving \eqref{HJ-eq-asymptotic}:
\begin{enumerate}
\item[(A)] $\displaystyle \lim_{x\to\infty}h_1(x)=1$, with
\[
\overline F_1(x)=e^{-x},
\qquad
\hat V_1(x)=e^{-x}(\sqrt{r+1}-\sqrt r)^2;
\]

\item[(B)] $\displaystyle \lim_{x\to\infty}h_2(x)=0$, with
\[
\overline F_2(x)=(1+x)^{-1}
\left(\frac{\sqrt r+\sqrt{r+(1+x)^{-1}}}{\sqrt r+\sqrt{r+1}}\right)^2,
\qquad
\hat V_2(x)=\frac{1}{(1+x)(\sqrt r+\sqrt{r+1})^2};
\]

\item[(C)] $\displaystyle \lim_{x\to\infty}h_3(x)=\infty$, with
\[
\overline F_3(x)=e^{-x-x^2/2}
\left(\frac{\sqrt r+\sqrt{r+1+x}}{\sqrt r+\sqrt{r+1}}\right)^2,
\qquad
\hat V_3(x)=\frac{e^{-x-x^2/2}}{(\sqrt r+\sqrt{r+1})^2};
\]

\item[(D)] $\displaystyle \lim_{x\to\infty}h_4(x)$ does not exist, with
\[
\overline F_4(x)=e^{-x-\frac14+\frac14\cos x}
\left(\frac{\sqrt r+\sqrt{r+1+\frac14\sin x}}{\sqrt r+\sqrt{r+1}}\right)^2,
\qquad
\hat V_4(x)=\frac{e^{-x-\frac14+\frac14\cos x}}{(\sqrt r+\sqrt{r+1})^2}.
\]
\end{enumerate}

We plot $\alpha_i^*(x)=\sqrt{r\hat V_i(x)/\overline F_i(x)}$, $i=1,2,3,4$, on $[0,50]$ using $10^4$ grid points, together with the corresponding optimal trajectories for the realization $\overline X(\omega)=10$. The trajectories are computed with MATLAB's \texttt{ode45} using $\mathrm{RelTol}=10^{-10}$ and $\mathrm{AbsTol}=10^{-12}$.

\begin{figure}[H]
\centering
\includegraphics[scale=0.15]{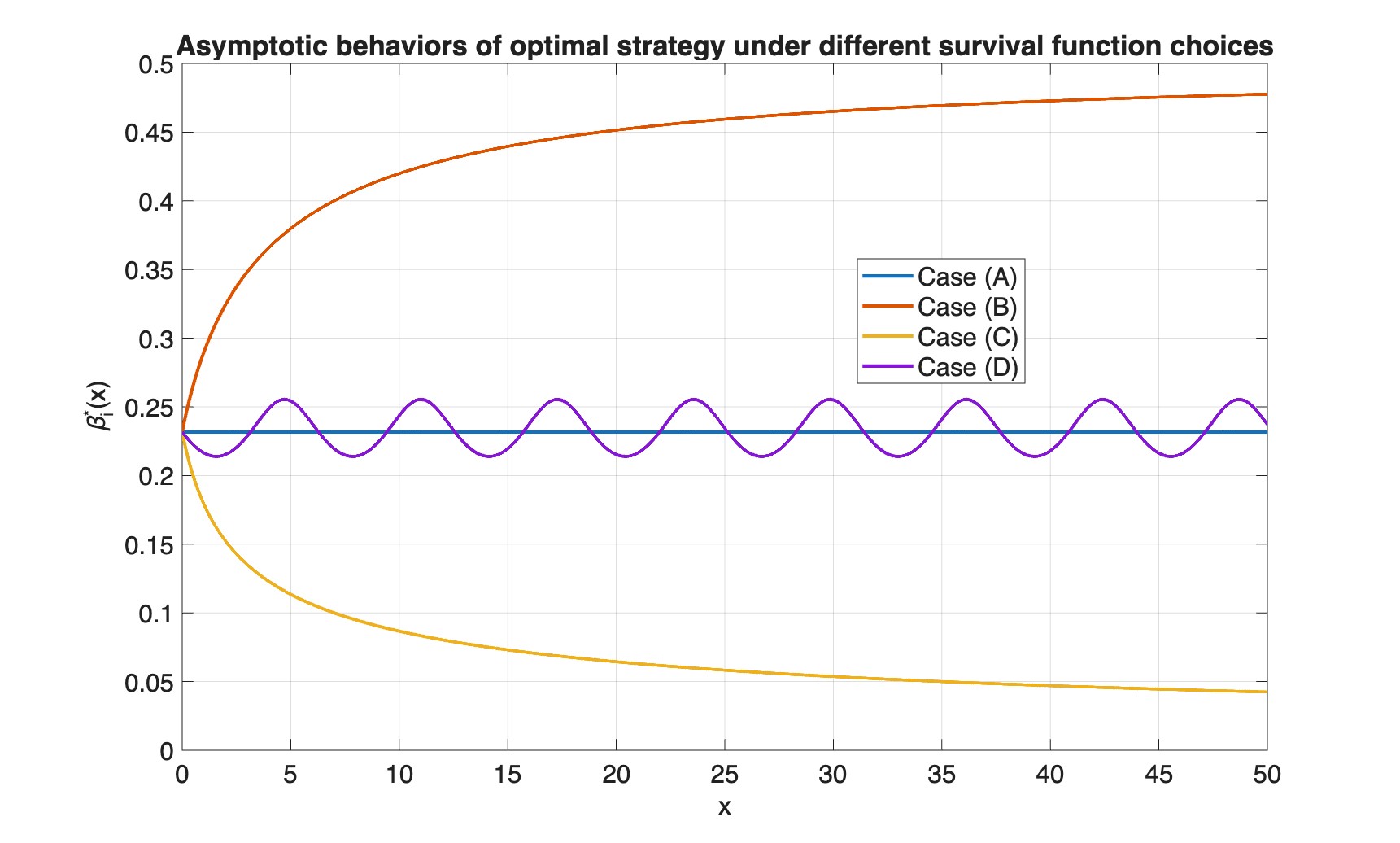}
\includegraphics[scale=0.15]{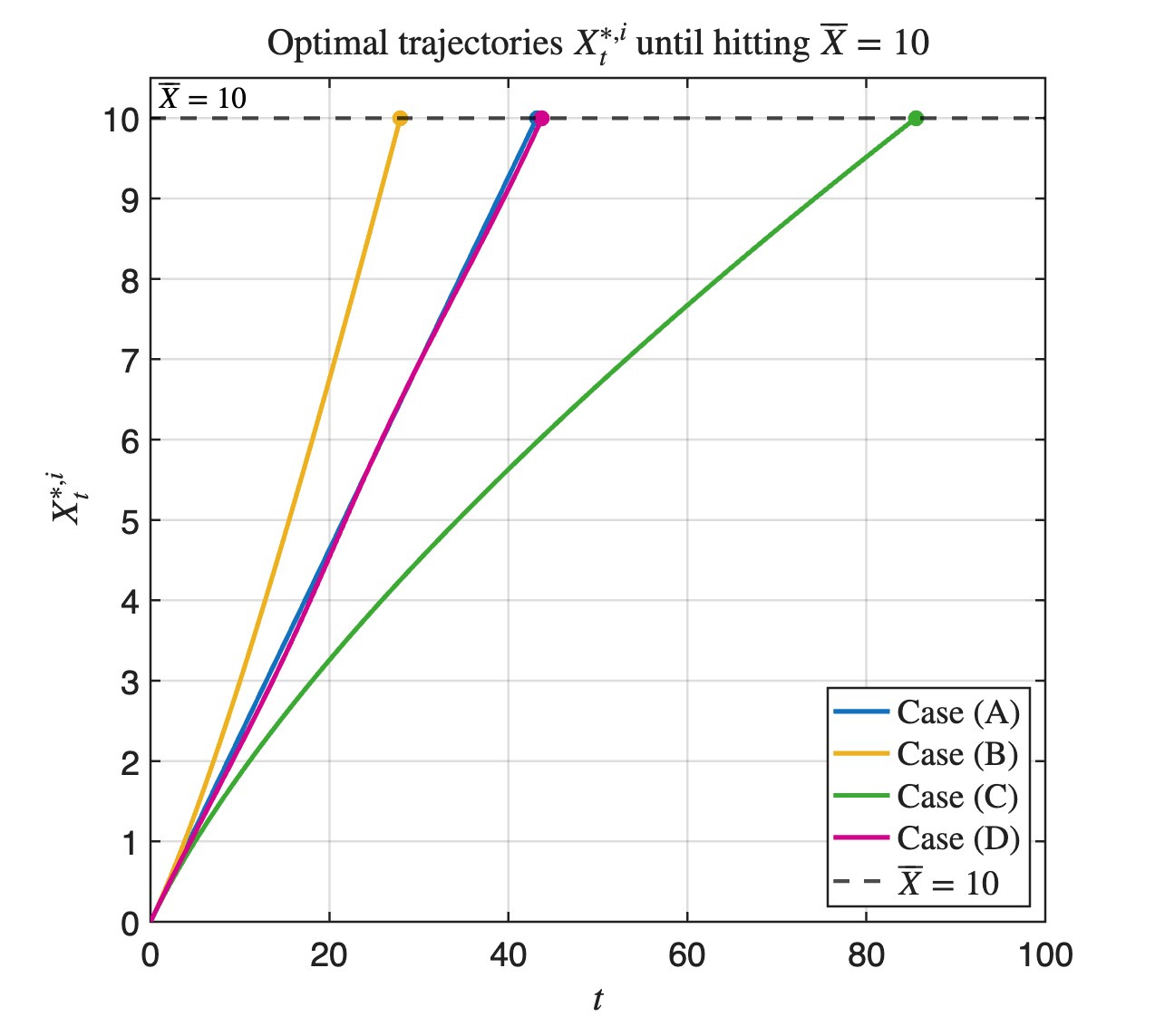}
\caption{Asymptotic behavior of optimal strategies (left) and corresponding optimal state trajectories (right).}
\label{fig:asymptotic_op_con_deterministic}
\end{figure}


{
\let\oldthebibliography\thebibliography
\renewcommand\thebibliography[1]{%
    \oldthebibliography{#1}%
    \setlength{\itemsep}{-1pt}%
    \setlength{\parskip}{0pt}%
}
\bibliographystyle{siam}
\bibliography{ref}
}

\end{document}